\documentclass[11pt,a4paper]{amsart}

\usepackage[left=26mm,right=26mm,top=25mm,bottom=25mm,
  headheight=14pt,headsep=7mm]{geometry}
\usepackage[section]{placeins}

\usepackage{mathrsfs}
\usepackage{amssymb}
\usepackage{microtype}
\usepackage{graphicx}
\usepackage{subcaption}

\usepackage{xcolor}
\usepackage{float}
\usepackage{needspace}
\usepackage[hidelinks]{hyperref}
\allowdisplaybreaks[1]
\newtheorem{theorem}{Theorem}[section]
\newtheorem{lemma}[theorem]{Lemma}

\newtheorem{corollary}[theorem]{Corollary}
\newtheorem{proposition}[theorem]{Proposition}
\theoremstyle{definition}
\newtheorem{definition}[theorem]{Definition}
\newtheorem{example}[theorem]{Example}

\theoremstyle{remark}
\newtheorem{remark}[theorem]{Remark}

\numberwithin{equation}{section}

\title[Convex-like Submanifolds, Caps and a Volume Comparison Theorem]
  {Convex-like Submanifolds, Caps\\
   and a Volume Comparison Theorem}
\author{Baichuan Hu}
\address{School of Mathematical Sciences, Peking University, Beijing 100871, P.R. China}
\curraddr{}
\email{2301110016@pku.edu.cn}
\thanks{}

\author{Xiang Ma}
\address{LMAM, School of Mathematical Sciences, Peking University, Beijing 100871, P.R. China}
\curraddr{}
\email{maxiang@math.pku.edu.cn}
\thanks{Corresponding author: Xiang Ma(maxiang@math.pku.edu.cn); partially supported by NSFC grant 11831005.}

\author{Dong Zhang}
\address{School of Mathematical Sciences, Peking University, Beijing 100871, P.R. China}
\curraddr{}
\email{dongzhang@math.pku.edu.cn}
\thanks{The third author was partially supported by NSFC grant 12401443.}

\begin{document}

\begin{abstract}
Convex-like submanifolds in the Lorentz-Minkowski spaces are introduced in this paper as a natural and broad generalization of convex surfaces in the Euclidean space. Associated with such a convex-like submanifold in a Minkowski space, we define roofs (which resembles the supporting planes), caps (which are generated by the light rays in the normal bundle), and towers (which revives the notion of interior for such high codimensional objects). As a touchstone of this theory, a volume comparison theorem is proved, which is a non-trivial generalization of the obvious fact that among two nested convex surfaces, the outer one has larger volume. Finally we provide two important applications of this theorem.
\end{abstract}

\maketitle

\section{Introduction}

Convex curves and surfaces in Euclidean spaces are the most classic geometric objects with complete theories and rich results \cite{Schneider}. Once turning to the Lorentz-Minkowski spaces, we would also like to find similar interesting and general enough geometric models. In particular, which kind of submanifolds are best generalization of these convex counterparts?

A naive attempt is looking at any bounded convex body with smooth boundary in the linear space (endowed with the Lorentz-Minkowski metric). It is obvious that the induced metric on the boundary could not be everywhere nondegenerate. Thus it is not a good candidate to do geometry.

But our hope is not in vain. If we consider a convex closed curve on a spacelike plane in $\mathbb{R}^3_1$ (the 3-dim Minkowski space), it can be characterized as a spacelike curve with a nowhere vanishing (spacelike) curvature vector. In other words, its osculating plane at each point is non-degenerate and spacelike everywhere. This is a nice property stable under slight perturbation, which allows us immediately to find a large class of similar objects comparable to convex loops. This intuition motivates the following generalization first in 2016.

\begin{definition}[\emph{strong spacelike}].

 \begin{enumerate}
      \item A smooth curve $\Gamma$ in $\mathbb{R}^3_1$ is called \emph{strong spacelike} if its tangent and curvature vector span a spacelike plane at each point on $\Gamma$ \cite{Ye1}.
       \item The index of a closed strong spacelike curve is the winding number of its unit tangent indicatrix on the de Sitter sphere, with the orientation chosen so that this number is positive. Equivalently, it is the rotation index of its projection onto a spacelike plane along any causal direction.
       \item A closed strong spacelike curve $\Gamma$ with index$=1$ is called \emph{convex-like}.
   \end{enumerate}
\end{definition}

In the 2016 paper \cite{Ye1} and the follow-up \cite{Ye2}, we showed that for such a convex-like loop $\Gamma$:
\begin{itemize}
    \item It projects to a closed convex planar curve under causal projections.
    \item $\Gamma$ spans a unique maximal (spacelike) surface \cite{Ye2}.
    \item The total curvature $\int_\Gamma \kappa ds \le 2\pi$ , with the equality attained if and only if $\Gamma$ is a convex plane curve.~~ (\emph{The reversed Fenchel-type inequality})
\end{itemize}

The first property above justifies the name \emph{convex-like}. Moreover, these results motivate us to ask the following much general questions:
\begin{enumerate}
    \item [Q1] What is the suitable generalization of \emph{strong spacelike} in the high dimensional (and higher codimensional) pseudo-euclidean spaces?
    \item [Q2] Without the smoothness assumption (like the polygon and polyhedron), how to generalize the definition of \emph{convex-like} $\Gamma$ in the general case?
    \item [Q3] How to generalize the basic constructions like supporting functions and supporting planes of convex objects to such $\Gamma$?
    \item [Q4] What is the possible physical meaning about these objects?
    \item [Q5] (global) In such a high codim$\ge 2$ case, can we generalize the basic \emph{inclusion/interior} relationship between such $\Gamma$'s and compare their volumes?
    \item [Q6] (global) Can we generalize those classical geometric inequalities to this case?
    \item [Q7] (local) Can we generalize the \emph{positive curvature properties} for these convex-like objects?
    \item [Q8] (global) What is the possible behavior of the convex-like submanifold $\Gamma$ under MCF (the mean curvature flow), or CSF (the curve shortening flow) for curves? In particular, will it converge to a circular limit, or create new types of singularities?
    \item [Q9] (global) The Plateau problem/ the Dirichlet problem. Namely, to establish the existence and uniqueness of a maximal submanifold spanned by such a convex-like boundary.
\end{enumerate}

These fundamental questions, as well as partial solutions to them, will constitute the core of the first author's PhD thesis. In this paper we just report partial results concerning Q1-Q5 above. \\

First we should point out that the \emph{strong spacelike} property has a natural generalization by Ben Andrews and Qiyu Zhou in \cite{AZ26} as follows.

\begin{definition} [\cite{AZ26}]\label{def-spacelike-convex}
An immersed spacelike submanifold $\Gamma^{n-1} \subset \mathbb{R}^{n+k}_k$ is \emph{spacelike-convex} if
\[
\langle h(v, v), h(v, v)\rangle > 0
\]
for all unit $v \in T\Gamma$, where $h$ is the second fundamental form.
\end{definition}

For Question~2, to include both smooth immersions and piecewise linear objects in our consideration and to overcome the difficulty of bad singularities, we use the projection images to characterize the generalized convexity as below.

\begin{definition}\label{def-convex-like}
An $(n-1)$-dimensional topological submanifold $\Gamma$ in $\mathbb{R}_1^{n+1}$ is \emph{convex-like} if for any causal projection $P = P_{v,\Sigma}$ (where the image hyperplane $\Sigma$ is spacelike, and $v$ is a timelike or null vector transversal to $\Sigma$), $P|_\Gamma $ is always a topological homeomorphism to the image $P(\Gamma)$, which is the boundary of a bounded open convex subset of $\Sigma$.
\end{definition}

\begin{remark}
We note that in 2025, Fukui, Honda and Umehara have introduced \emph{essentially cocausal submanifolds} in \cite{Fukui}, whose significance can be seen from the title of this paper:
\emph{Convexity of space-like projections of
submanifolds with co-dimension 2 in
Lorentz–Minkowski space}.
Such a submanifold is characterized in terms of its first and second fundamental form at each point (see Definition~\ref{def-essentially-cocausal}). This property implies, under any timelike orthogonal projection, such a submanifold will correspond to a locally strictly convex hypersurface~\cite{Fukui}.

Compared with this notion, our \emph{convex-like} property in Definition~\ref{def-convex-like} greatly weakens smoothness assumptions, at the same time requiring that the tangent plane at any immersed point is still spacelike; moreover, the projection image of a convex-like submanifold is not strictly convex in general. Thus any of these two notions does not imply the other. Further discussions are left to Section~3. (See Example~\ref{example}.)
\end{remark}

\begin{remark}
Note that \emph{spacelike convexity} in Definition~\ref{def-spacelike-convex} implies \emph{essentially cocausal} property in Defitinion~\ref{def-essentially-cocausal}.

Below in Section~3 we will show that \emph{spacelike-convexity} also implies \emph{convex-like} property if the submanifold has dimension greater than 1 and is assumed to be smooth, connected, compact without boundary. (See Proposition~\ref{prop-spacelike-convex-like} and Remark~\ref{remark-index-1}.)
\end{remark}

After defining a convex-like submanifold $\Gamma$, we introduce four important notions. The first one is \emph{roof}, namely a pair of null supporting planes at every point of $\Gamma$. Second, we define their \emph{support functions} using only null directions. Third, \emph{towers} are introduced as one-sided Lorentzian analogues of the regions enclosed by convex hypersurfaces, with the two orientations corresponding to the past and the future. The intersection of the two towers associated with a convex-like submanifold is called its \emph{acausal interior}, an analogue of the convex body enclosed by a convex hypersurface in the Euclidean case. Finally, we define a \emph{cap} as the boundary of a tower, which is also realized as non-smooth null hypersurfaces carrying convex-like submanifolds.

\begin{remark}
Our previous question Q4 is about the physical meaning. In the regard, it is not difficult to see that a convex-like submanifold is acausal in the usual spacetime language. See Proposition~\ref{prop-acausal}. Furthermore, \emph{acausal interior} of a convex-like $\Gamma$ is consist of all points having NO causal relationship with $\Gamma$,

A notion related to caps and our convex-like property is \emph{past null convexity}, which arises in the study of the Penrose inequality in Minkowski spacetime~\cite{BW14}.
\end{remark}

These constructions helps us to compare the relative positions of convex-like submanifolds through a notion of \emph{nesting}. Roughly speaking, $\Gamma_1$ nests in $\Gamma_2$ if, under every causal projection $P$, $P(\Gamma_1)$ is enclosed by $P(\Gamma_2)$. In Euclidean convex geometry, inclusion of convex bodies implies monotonicity of their boundary volumes. Our main result extends this volume comparison principle to convex-like submanifolds.

\begin{theorem}\label{volume comparison}
Let $\Gamma_1$ and $\Gamma_2$ be two convex-like submanifolds in $\mathbb{R}_1^{n+1}$. If $\Gamma_1$ lies in the acausal interior of $\Gamma_2$, then
\[
V^{n-1}(\Gamma_1) \leq V^{n-1}(\Gamma_2).
\]
\end{theorem}

\begin{remark}
The notion of convex-like submanifolds extends naturally to pseudo-Euclidean spaces of higher negative index. More precisely, an $(n-1)$-dimensional submanifold of $\mathbb{R}^{n,k}$ is called convex-like if every projection along a maximal negative semidefinite subspace restricts to a homeomorphism onto an affine convex hypersurface. For such submanifolds, we can still define the acausal interior and the nesting relation. However, roofs, support functions, and caps become much more difficult to study because of the additional timelike directions.

The cap method developed in this paper does not extend to higher negative index. Inspite of that, we can establish the volume comparison theorem for arbitrary negative index using geometric deformations and approximations. This approach will be developed in a separate paper.
\end{remark}

Although Q6--Q9 lie beyond the scope of this paper, a number of results concerning these questions are already available.

\begin{remark}
\textbf{For Q6}, a Fenchel-type inequality was obtained jointly by the second author, Nan Ye, and Donghao Wang; see \cite{Ye1,Ye2}.

\textbf{For Q7}, in unpublished work we have proved that smooth spacelike convex-like submanifolds have nonnegative sectional curvature and that convex-like polyhedrons have nonnegative Regge curvature. In the smooth spacelike-convex setting, Andrews and Zhou proved positivity of sectional curvature; see \cite[Lemma~3.7]{AZ26}.

\textbf{For Q8}, relevant works include the following:

Andrews and Zhou~\cite{AZ26} prove that compact connected spacelike-convex submanifolds without boundary, with one spacelike codimension, shrink to a point under mean curvature flow and converge smoothly to a round sphere after suitable rescaling. In the curve case, turning number one is additionally assumed.

The first and second authors~\cite{HM26} establish a dichotomy for closed smooth spacelike curves admitting a one-to-one convex projection onto a spacelike plane: at the finite maximal existence time, either the curve shrinks to a point and becomes asymptotically circular, or its tangent directions approach the null cone along a sequence. Degeneration to null tangent directions under the flow cannot occur for convex-like curves, but may occur for closed strong spacelike curves of higher index or for general spacelike curves admitting a one-to-one convex projection.

\textbf{For Q9}, a closely related result is Li's work on the Dirichlet problem for maximal graphs of higher codimension
\cite{Li22}.
\end{remark}

The contents of this article are arranged as follows. \textbf{Section~2} contains some basic notation about Lorentzian geometry. In \textbf{Section~3}, we give the definition of convex-like submanifolds and their null supports, and provide some important properties. In \textbf{Section~4}, we introduce the notions of towers and caps. In \textbf{Section~5}, we define the relation of \emph{nest in}. In \textbf{Section~6}, we recall the Euclidean volume comparison theorem and give two standard proofs, both of which meet great trouble when generalized to Minkowski spacetime.

In \textbf{Section~7}, we prove the Volume Comparison Theorem~\ref{volume comparison}: for two convex-like submanifolds $A$ and $B$, if $A$ \emph{nests in} $B$ then $A$ has no greater volume than $B$. The basic idea is to reduce the problem to comparing two sections (which might be quite irregular) of the same \emph{cap}, in which case we can easily show that the correspondence between them along the light rays are monotonically distance increasing/decreasing. As to the volumes of geometric objects (not necessarily smooth), they are hard to define and measure directly using Lorentz-Minkowski geometry. We overcome this difficulty by establishing bi-Lipschitz mappings to round spheres, which makes the rest part of this section unavoidably quite technical.

As applications, in \textbf{Section~8}, we prove that convergence of the caps' support functions implies convergence of their intersections and that volume is continuous under Hausdorff convergence of convex-like submanifolds. In \textbf{Section~9}, we use the volume comparison theorem to show that, for a convex-like submanifold in $\mathbb{R}_1^{n+1}$, its projection volume ratio is bounded by $n^{n-1}$. Some conclusions about Lipschitz maps are collected in \textbf{Appendix~A}.

\vspace{0.5cm}

\noindent\textbf{Statement of AI use.}

The authors used ChatGPT to assist in language editing and in reviewing the proofs.
In particular, the proof of Proposition \ref{convex-like smooth equivalent}, Theorem \ref{euclidean volume comparison}, Proposition \ref{convex-like lip to sphere} and Lemma \ref{central projection lip}  was refined with the assistance of ChatGPT.

The authors take full responsibility for the content of this manuscript.

\section{Basic Notations}
For background on semi-Riemannian and Lorentzian geometry,
we refer to \cite{ONeill}.

Throughout the main text, we assume $n\geq2$.
The Lorentz-Minkowski space $\mathbb{R}_1^{n+1}$ is the $(n+1)$-dimensional real linear space endowed with the Lorentzian inner product $\langle\cdot ,\cdot\rangle$ and a fixed time-orientation. In a fixed Lorentzian coordinate system, we may write
\[
\langle X ,Y \rangle= \sum_{i = 1}^n x_iy_i  - x_{n+1}y_{n+1}, ~~~~X = (x_1, ..., x_{n+1}), Y = (y_1, ..., y_{n+1}).
\]

Also, in a fixed coordinate, we denote by $\langle\cdot ,\cdot\rangle_E$ the Euclidean inner product:
\[
\langle X ,Y \rangle_E= \sum_{i = 1}^{n+1} x_iy_i,\,\, ~~~~X = (x_1, ..., x_{n+1}), Y = (y_1, ..., y_{n+1}).
\]

Denote the Euclidean length of vectors by:
\[
|v|_E = \sqrt{\langle v ,v \rangle_E},
\]
and let $\mathrm{dist}(x,y) = |x-y|_E$ be the Euclidean distance between $x$ and $y$. We always use the Euclidean topology, and different choices of Lorentzian coordinates define the same topology.

In the subsequent discussions, unless otherwise specified, all coordinate systems are assumed to be Lorentzian orthogonal, with the last coordinate component representing time. If a coordinate is fixed, then we can consider the Euclidean properties of any subset or submanifold by regarding the coordinate as a Euclidean one.

For a subset $\Gamma\subset \mathbb{R}_1^{n+1}$, denote its convex hull as usual by
\[
\mathrm{Conv}(\Gamma) =\left\{\sum_{i=1}^{k}\lambda_ix_i|k\in \mathbb{N^*},x_i\in \Gamma,\lambda_i\geq 0,\sum_{i=1}^{k}\lambda_i = 1\right\}.
\]

For a linear subspace (or nonzero vector) $V$, let $V^\perp$ and $V^{\perp_E}$ be the Lorentzian and Euclidean orthogonal complements of $V$. Notice that the Euclidean length, distance and orthogonal complement all depend on the choice of Lorentzian coordinate system.

Recall that in the Lorentz space $\mathbb{R}_1^{n+1}$, a vector $X$ is called \emph{spacelike (null, timelike)} if $\langle X ,X \rangle > 0(= 0, < 0)$, respectively. A linear subspace $V$ is called \emph{spacelike (null, timelike)} if the Lorentz inner product restricts to a Riemannian (degenerate, Lorentz) inner product on $V$. An affine subspace can be of these three types similarly.

\textbf{We use the word \emph{causal} to describe an object being \emph{either timelike or null}, and \emph{cocausal} to represent \emph{spacelike or null} ones, following \cite{Fukui}. (In particular, a null vector is both causal and cocausal.)}

For a cocausal vector $v$, we write
\[
|v|_L=\sqrt{\langle v,v\rangle}.
\]

If $x = (x_1, ..., x_{n+1})$ is causal and $x_{n+1} > 0$ ($x_{n+1} < 0$), we call $x$ a future-pointing (past-pointing) vector. We denote
\begin{align*}
\mathrm{Timecone}^+(x) &= \{y~|~y-x\text{ is future-pointing and timelike}\};\\
\mathrm{Timecone}^-(x) &= \{y~|~y-x\text{ is past-pointing and timelike}\};\\
\mathrm{Nullcone}^+(x) &= \{y~|~y-x\text{ is future-pointing and null}\};\\
\mathrm{Nullcone}^-(x) &= \{y~|~y-x\text{ is past-pointing and null}\}.
\end{align*}
We may omit the position vector $x$ if $x=0$, for example,
\[
\mathrm{Nullcone}^+ = \mathrm{Nullcone}^+(0) = \{\text{future-pointing null vectors}\}.
\]

Topologically, $\mathrm{Nullcone}^\pm(x) \cup \{x\}$ is the boundary of $\mathrm{Timecone}^\pm(x)$.

Let $\pi$ be a cocausal hyperplane, then $\pi$ divides $\mathbb{R}_1^{n+1}$ into two halfspaces. One of them contains $\pi + (0,0,...,1)$ and we call it the \emph{future halfspace} with respect to $\pi$ (or briefly say $\pi$'s \emph{future}). Similarly, the other halfspace is said to be the \emph{past halfspace} with respect to $\pi$ (or $\pi$'s \emph{past}). A set $\Sigma \subset \mathbb{R}_1^{n+1}$ is said to be \emph{above} \emph{(below)} $\pi$ if every point $x\in \Sigma$ is in $\pi$ or in $\pi$'s future (past).

A subset $W\subset \mathbb{R}_1^{n+1}$ is said to be \emph{bounded}, if for any $v\in \mathbb{R}_1^{n+1}$, the set $\{\langle v,w\rangle\big |w\in W\}$ is bounded. This is equivalent to saying $|w|_E$ is bounded in some fixed coordinate.

\begin{remark} We shall note that the above definition applies only to specific causal characters; in other words, we do not distinguish being future-pointing or past-pointing for spacelike vectors, and do not distinguish future or past for corresponding halfspaces of any timelike hypersurfaces. So the above definitions about future and past only depend on time-orientation and do not depend on the choice of coordinate system.
\end{remark}

Given an n-dimensional spacelike subspace $\Sigma$ and a vector $v \notin \Sigma$, we denote by $P_{v,\Sigma}$ the parallel projection to $\Sigma$ along $v$.

If $v$ is fixed, then for $\Sigma_1 \neq \Sigma_2 $ there exists an affine map $\phi:\Sigma_1 \rightarrow \Sigma_2$ such that $\phi \circ P_{v,\Sigma_1} = P_{v,\Sigma_2}$. So if we consider affine invariant properties, we may not choose a specific $\Sigma $ and only write $P_v$.

If $v$ is timelike (null, spacelike), then we say $P_v$ is a \emph{timelike (null, spacelike) projection}; and if $v \perp \Sigma$, we say $P_{v,\Sigma}$ is an \emph{orthogonal projection}.

\begin{remark} In most cases, the orthogonal projection is useful and sufficient. However, if $v$ is null then there is no orthogonal projection along $v$ so we have to use other parallel projections.
\end{remark}

\section{Convex-like submanifolds}

For smooth spacelike submanifolds, the following proposition gives an equivalent local characterization of the convex-like property.

\begin{proposition}\label{convex-like smooth equivalent}
Let $\Gamma^{n-1} \subset \mathbb{R}_1^{n+1}$, $n\ge 3$, be a smooth,
compact, connected spacelike submanifold without boundary.
Let $h$ denote its vector-valued second fundamental form, with
the convention
\[
h(v,w)=(D_vw)^\perp,
\]
where $D$ is the ambient flat connection.
Then the following conditions are equivalent:
\begin{enumerate}
    \item $\Gamma$ is convex-like.
    \item For every $p\in\Gamma$ and $v\in T_p\Gamma$,
    \[
    \langle h(v,v),h(v,v)\rangle\ge 0,
    \]
    and there exists a smooth spacelike normal vector field $N$
    along $\Gamma$ such that
    \[
    \langle h(v,v),N\rangle\ge 0
    \qquad
    \text{for every }p\in\Gamma,\ v\in T_p\Gamma.
    \]
\end{enumerate}
\end{proposition}

\begin{proof}
First suppose that $\Gamma$ is convex-like.
Fix a future-pointing unit timelike vector $e_0$, and let
$P_0$ be the orthogonal projection onto the spacelike hyperplane
$H_0=e_0^\perp$.
Since $\Gamma$ is spacelike, $P_0|_\Gamma$ is an immersion.
By the convex-like assumption, it is therefore a smooth embedding
onto the boundary of a convex body in $H_0$.

Let $N_p\in H_0$ be the inward unit normal of this convex
hypersurface at $P_0(p)$, regarded as an ambient vector.
Then $N$ is a smooth unit spacelike normal field along $\Gamma$.
Indeed, for $v\in T_p\Gamma$,
\[
\langle N_p,v\rangle
=
\langle N_p,dP_0(v)\rangle_E
=0,
\]
where $\langle\cdot,\cdot\rangle_E$ is the Euclidean inner product
on $H_0$ induced by the Lorentzian metric.
Convexity also gives
\[
\langle q-p,N_p\rangle\ge 0
\qquad\text{for every }q\in\Gamma,
\]
and hence
\[
\langle h(v,v),N_p\rangle\ge 0.
\]

Let $T$ be the future-pointing unit timelike normal field
orthogonal to $N$.
Thus
\[
\langle N,N\rangle=1,\qquad
\langle T,T\rangle=-1,\qquad
\langle N,T\rangle=0.
\]
Fix $p\in\Gamma$ and $v\in T_p\Gamma$, and write
\[
h(v,v)=aN_p+bT_p.
\]
We show that $a\ge |b|$.

For $\theta\in\mathbb{R}$, define the constant ambient vectors
\[
u_\theta=\sinh\theta\,N_p+\cosh\theta\,T_p,
\qquad
\zeta_\theta=\cosh\theta\,N_p+\sinh\theta\,T_p.
\]
Then $u_\theta$ is unit timelike, $\zeta_\theta$ is unit
spacelike, and
\[
\langle u_\theta,\zeta_\theta\rangle=0.
\]
Let $P_\theta$ be the orthogonal projection onto
$u_\theta^\perp$.
The vector $\zeta_\theta$ is a unit normal to
$P_\theta(\Gamma)$ at $P_\theta(p)$.

We verify that $\zeta_\theta$ is the inward normal for every
$\theta$.
Let
\[
c=\frac{1}{\operatorname{Vol}(\Gamma)}
  \int_\Gamma q\,dV_\Gamma(q),
\]
be the barycenter of $\Gamma$, where $dV_\Gamma$ is the volume measure of the induced
Riemannian metric.
For every $\theta$, the point $P_\theta(c)$ must lie in the
interior of the convex body bounded by $P_\theta(\Gamma)$. Consequently,
\[
\langle c-p,\zeta_\theta\rangle\ne 0
\qquad\text{for every }\theta.
\]
This expression depends continuously on $\theta$ and is
positive at $\theta=0$, so it remains positive for all $\theta$, proving that
$\zeta_\theta$ is inward-pointing.
Thus
\[
\langle q-p,\zeta_\theta\rangle\ge 0
\qquad\text{for every }q\in\Gamma.
\]
Taking the second derivative along a smooth curve in
$\Gamma$ through $p$ with initial velocity $v$ gives
\[
0\le
\langle h(v,v),\zeta_\theta\rangle
=
a\cosh\theta-b\sinh\theta.
\]
Dividing by $\cosh\theta$ and letting $\theta\to\pm\infty$,
we obtain
\[
a-b\ge 0,\qquad a+b\ge 0.
\]
Hence $a\ge |b|$, and therefore
\[
\langle h(v,v),h(v,v)\rangle=a^2-b^2\ge 0.
\]

Conversely, suppose that condition~(2) holds.
After normalization, assume that $\langle N,N\rangle=1$,
and let $T$ be the future-pointing unit timelike normal
field orthogonal to $N$.
For every tangent vector $v$, write
\[
h(v,v)=a(v)N+b(v)T.
\]
The assumptions imply
\[
a(v)\ge 0,\qquad a(v)^2-b(v)^2\ge 0,
\]
so that
\[
a(v)\ge |b(v)|.
\]

Fix any nonzero future-pointing causal vector $u$, and let
$P_u$ be the parallel projection along $\mathbb{R}u$ onto
a fixed spacelike hyperplane $H_0$.
Set
\[
Y=P_u|_\Gamma.
\]
Since no nonzero causal vector is tangent to $\Gamma$,
the map $Y$ is an immersion.

Decompose $u$ into its tangential and normal components:
\[
u=u^\top+\alpha N+\beta T,
\qquad
\alpha=\langle u,N\rangle,\qquad
\beta=-\langle u,T\rangle>0.
\]
Because $u$ is causal and $T\Gamma$ is positive definite,
\[
\beta^2-\alpha^2
=
\langle u^\top,u^\top\rangle-\langle u,u\rangle
\ge 0.
\]
In particular, $\beta\ge|\alpha|$.
Define
\[
\zeta=\beta N+\alpha T.
\]
Then $\zeta$ is nonzero and satisfies
\[
\langle\zeta,u\rangle=0,
\qquad
\langle\zeta,w\rangle=0
\quad\text{for every }w\in T\Gamma.
\]

Let $Z\in H_0$ be the Euclidean representative of the
restriction of the covector $\langle\cdot,\zeta\rangle$
to $H_0$; namely,
\[
\langle z,Z\rangle_E=\langle z,\zeta\rangle
\qquad\text{for every }z\in H_0.
\]
The vector $Z$ is nonzero: otherwise $\zeta$ would annihilate
both $H_0$ and $u$, and hence the entire ambient space.
Moreover, $Z$ is normal to the immersion $Y$, since
$P_uw-w$ is a multiple of $u$.

Let $h_Y$ denote the Euclidean second fundamental form of
$Y$, viewed as a bilinear form on $T\Gamma$.
The unit normal $\widehat N=Z/|Z|_E$ satisfies
\[
\begin{aligned}
\langle h_Y(v,v),\widehat N\rangle_E
&=\frac{\langle h(v,v),\zeta\rangle}{|Z|_E}\\
&=\frac{\beta a(v)-\alpha b(v)}{|Z|_E}\\
&\ge
\frac{\beta a(v)-|\alpha|\,|b(v)|}{|Z|_E}\\
&\ge 0.
\end{aligned}
\]
Thus $Y$ has positive semidefinite second fundamental form
with respect to a smooth global unit normal field.
It follows that $Y$ is locally convex: in a local graph
representation, the corresponding Hessian is positive
semidefinite.

Since $\Gamma$ is compact, the Euclidean metric induced
by $Y$ is complete.
Furthermore, a sphere enclosing $Y(\Gamma)$ has a point
of contact at which $Y$ is locally strictly convex.
By van Heijenoort's global convexity theorem
\cite{Van52}, since
$n-1\ge 2$ and $\Gamma$ is connected, $Y$ is an embedding
onto the boundary of a convex body.
\end{proof}

Consequently, any compact, connected spacelike-convex submanifold in $\mathbb{R}_1^{n+1}$ ($n\geq 3$) without boundary must be convex-like.

\begin{proposition}\label{prop-spacelike-convex-like}
Let $\Gamma^{n-1} \subset \mathbb{R}_1^{n+1}$ be a compact, connected spacelike-convex submanifold without boundary, where $n\geq 3$. Then $\Gamma$ is convex-like.
\end{proposition}

\begin{proof}
For every $p\in \Gamma$ and $0\neq u,v\in T_p\Gamma$, $h(v,v),h(u,u)$ are two spacelike vectors in $(T_p\Gamma)^\perp \cong \mathbb{R}_1^2$. Hence $\langle h(v,v),h(u,u)\rangle \neq 0$. By continuity,
\[
\langle h(v,v),h(u,u)\rangle > 0.
\]

Choose $N(p) = H(p) =  \sum h(e_i,e_i)$ (where $\{e_i\}$ is an orthonormal basis of $T_p\Gamma$) be the mean curvature vector, then $H$ is spacelike and $\langle H,h(v,v)\rangle > 0$ for each $v$. Proposition \ref{convex-like smooth equivalent} completes the proof.
\end{proof}

\begin{remark}\label{remark-index-1}
For the case where $n=2$, we refer to \cite[Lemma 2.1]{Ye2}, which shows that a closed strong spacelike curve with index $1$ is convex-like.
\end{remark}

Compared to spacelike convexity, the notion of essentially cocausal immersion given by Fukui, Honda and Umehara \cite[Definition 1]{Fukui} replaces the spacelikeness of tangent space and curvature vector by the cocausal property.

\begin{definition}\label{def-essentially-cocausal}
Let $F:M^{n-1} \to \mathbb{R}_1^{n+1}$ be a smooth immersion. Then we say F is an essentially cocausal immersion if:

(a)$dF(T_pM)$ is cocausal for each $p \in M$;

(b)for any point $p\in M$, any nonzero tangent vector $X\in T_pM$ and any spacelike normal vector $n$ at $F(p)$, there holds
\[
\langle\nabla_XdF(\widetilde X),n\rangle\neq0,
\]
where $\widetilde X$ is any local tangent vector field with $\widetilde X(p)=X$. The displayed normal quadratic form is independent of the extension $\widetilde X$.
\end{definition}

They proved that an essentially cocausal immersed submanifold projects onto a locally strictly convex hypersurface under any timelike projection \cite[Theorem 4]{Fukui}. However, such submanifolds may admit null tangent directions and hence are not necessarily convex-like (see Theorem~\ref{uniformly spacelike})

To better understand the relationships and distinctions among these notions, we give the following examples.

\begin{example}\label{example}

(a)\textbf{Submanifolds in hyperplanes}. Suppose $\Sigma$ is a hyperplane in $\mathbb{R}_1^{n+1}$ and $\Gamma$ is a closed convex hypersurface in $\Sigma$.

$\Gamma$ is convex-like if and only if $\Sigma$ is spacelike, not requiring the smoothness of $\Gamma$.

Suppose $M^{n-1}$ is a Riemannian manifold and $i:M^{n-1}\to\Gamma$ is a smooth immersion. Then $i$ is an essentially cocausal immersion if and only if $\Sigma$ is cocausal and the second fundamental form of $i$ as a hypersurface of $\Sigma$, computed using any Euclidean inner product on $\Sigma$, is definite at every point of $M$. The definiteness condition is equivalently expressed by the following: for every local tangent vector field $X$ with $X(p)\neq0$,
\[
\nabla_{X(p)}di(X)\notin di(T_pM).
\]

Let $\Gamma$ be smooth. Then $\Gamma$ is spacelike convex if and only if $\Sigma$ is spacelike and the second fundamental form of $\Gamma$ is definite at each point of $\Gamma$.

(b)\textbf{Polygons in $\mathbb{R}_1^3$}. If a closed polygon in $\mathbb{R}_1^3$ is convex-like, then each of its edges $e_i$ is spacelike and each $e_i \times_L e_{i+1}$ is zero or causal. Moreover, let $v$ be any timelike vector, then $\langle e_i \times_L e_{i+1},v\rangle$ and $\langle e_j \times_L e_{j+1},v\rangle$ cannot have opposite signs. Here $\times_L$ denotes the Lorentzian outer product, i.e.,
\[
(x_1,x_2,x_3) \times_L(y_1,y_2,y_3) = (x_3y_2-x_2y_3,x_1y_3-x_3y_1,x_1y_2-x_2y_1).
\]

(c)\textbf{A family of nonplanar curves}.
Consider the closed curves
\[
\Gamma_a(\theta)
=
(\cos\theta,\sin\theta,a\cos 2\theta),
\qquad
\theta\in\mathbb{R}/2\pi\mathbb{Z},
\qquad
0<a<\frac12,
\]
in $\mathbb{R}_1^3$.
Each $\Gamma_a$ is a smooth, uniformly spacelike embedded curve
whose vertical projection is the unit circle. Simple calculation yields:
\begin{enumerate}
\item If $0<a<\frac14$, then $\Gamma_a$ is strong spacelike,
 essentially cocausal, and convex-like.
\item If $a=\frac14$, then $\Gamma_a$ is essentially cocausal
and convex-like, but is not strong spacelike.
\item If $\frac14<a<\frac12$, then $\Gamma_a$ is not strong spacelike,
not essentially cocausal, and not convex-like, despite being uniformly spacelike
and having a convex vertical projection.
\end{enumerate}
Thus, convex-like curves may have null osculating planes,
whereas uniform spacelikeness and convexity of a single
projection do not imply the convex-like property.
\end{example}


The following theorem is a generalization of the Section Lemma \cite[ Lemma 2.2]{Ye2}.

\begin{theorem}\label{intersection with timelike hyperplane}
Let $\Gamma^{n-1}$ be a convex-like submanifold of $\mathbb{R}_1^{n+1}$ and $\Sigma$ is a timelike hyperplane, where $n\geq 3$. If $\Gamma$ crosses $\Sigma$ $($i.e. $\Gamma$ lies in both sides of $\Sigma$$)$, then $\Gamma \cap \Sigma$ is a convex-like submanifold of $\Sigma$.
\end{theorem}

\begin{proof}
For any causal vector $v \parallel \Sigma$ and any $(n-1)$-dimensional spacelike subspace $\Sigma'$ not containing $v$, consider the projection $P_{v,\Sigma'}:\Sigma\to \Sigma'$. It is the restriction of a causal projection $\hat{P}_{v,\Sigma_1}:\mathbb{R}_1^{n+1}\to \Sigma_1$, where $\Sigma_1$ is properly chosen so that $\Sigma' = \Sigma \cap \Sigma_1$. $\hat{P}_{v,\Sigma_1}|_{\Gamma}$ is a homeomorphism from $\Gamma$ to the convex hypersurface $\hat{P}_{v,\Sigma_1}(\Gamma) \subset \Sigma_1$, so
\[
P_{v,\Sigma'}|_{\Gamma \cap \Sigma} = \hat{P}_{v,\Sigma_1}|_{\Gamma \cap \Sigma}
\]
is a homeomorphism from $\Gamma \cap \Sigma$ to the convex hypersurface (guaranteed by $\Gamma$ crossing $\Sigma$) $\hat{P}_{v,\Sigma_1}(\Gamma)\cap\Sigma \subset \Sigma'=\Sigma_1\cap\Sigma$.
\end{proof}

A fundamental and useful theorem about convex bodies is the existence of supporting hyperplane. Correspondingly, convex-like submanifolds are supported by a pair of null hyperplanes at each point.

\begin{definition}
Given a set $\Sigma \subset \mathbb{R}_1^{n+1}$ and a point $x \in \Sigma$. A \emph{null roof} (or simply a \emph{roof}) of $\Sigma$ at $x$ is a pair of different null hyperplanes $(\pi^+, \pi^-)$ that satisfy:

(a)$x \in \pi^+ \cap \pi^-$;

(b)$\Sigma $ is below $\pi^+$ and above $\pi^-$.

If the pair exists, we call $\pi^+$ an \emph{upper null support} of $\Sigma$ at $x$, and $\pi^-$ a \emph{lower null support} of $\Sigma$ at $x$.

\end{definition}

\begin{theorem}\label{existence of roof}
Let $\Gamma$ be an $(n-1)$-dimensional topological submanifold in $\mathbb{R}_1^{n+1}$. If $\Gamma$ is convex-like, then for every $x \in \Gamma$ there exists a null roof of $\Gamma$ at $x$.

\end{theorem}

\begin{proof}

For any causal projection $P$, $P$ is linear so
\[P(\mathrm{Conv}(\Gamma)) = \mathrm{Conv}(P(\Gamma)),\]
hence
\[P(\Gamma) = \partial \mathrm{Conv}(P(\Gamma)) = \partial P(\mathrm{Conv}(\Gamma)).\]

Take a coordinate system $(x^1,...,x^{n+1})$ and let $x$ be any point in $\Gamma$.

If $\mathrm{Conv}(\Gamma)$ is not full-dimensional, then $\Gamma$ lies in a spacelike hyperplane, and the conclusion is immediate. Hence we may assume that $\mathrm{Conv}(\Gamma)$ has nonempty interior.

Claim: $\mathrm{Conv}(\Gamma)$ does not intersect $\mathrm{Timecone}^+(x)$.

If $\mathrm{Conv}(\Gamma)$ intersects $\mathrm{Timecone}^+(x)$ then there exists a Euclidean ball $B(a,r) \subset \mathrm{Conv}(\Gamma) \cap \mathrm{Timecone}^+(x)$. Consider the timelike projection $P$ along $a-x$, then $P(x)$ is in the interior of $(P(\mathrm{Conv}(\Gamma))) $, which contradicts to $P(\Gamma) = \partial P(\mathrm{Conv}(\Gamma))$.

Step 1: Find a supporting hyperplane of $\mathrm{Conv}(\Gamma)$ at $x$.

Without loss of generality suppose $x = (0,...,0)$. Consider a sequence of $n$-dimensional closed disks:
\[
D_k = \{(x_1,...,x_n,1/k)|x_1^2+...+x_n^2 \leq (\frac{k-1}{k^2})^2\}.
\]
Then $D_k$ does not intersect $\mathrm{Conv}(\Gamma)$.

Suppose $p_k \in \mathrm{Conv}(\Gamma)$ and $q_k \in D_k$ satisfy
\[
\mathrm{dist}(p_k,q_k) = \mathrm{dist}(\mathrm{Conv}(\Gamma),D_k),
\]
let $v_k = \frac{q_k-p_k}{\mathrm{dist}(p_k,q_k)}$ and $v$ be an accumulation point of $\{v_k\}$. We calculate that
\begin{align*}
    \mathrm{dist}(v_k,x) &\leq \mathrm{dist}(v_k,v_k+p_k)+\mathrm{dist}(v_k+p_k,p_k)+\mathrm{dist}(p_k,x)\\
&= \mathrm{dist}(v_k,v_k+p_k)+\mathrm{dist}(v_k+p_k,\mathrm{Conv}(\Gamma))+\mathrm{dist}(p_k,x)\\
&\leq 2\mathrm{dist}(v_k,v_k+p_k)+\mathrm{dist}(v_k,\mathrm{Conv}(\Gamma))+\mathrm{dist}(p_k,x)\\
&= 3\mathrm{dist}(p_k,x) + \mathrm{dist}(v_k,\mathrm{Conv}(\Gamma)) \\
&\leq 3(\mathrm{dist}(p_k,q_k)+\mathrm{dist}(q_k,x)) + \mathrm{dist}(v_k,\mathrm{Conv}(\Gamma))\\
&\leq 6\mathrm{dist}(q_k,x) + \mathrm{dist}(v_k,\mathrm{Conv}(\Gamma)) \\
&\leq \frac{12}{k} + \mathrm{dist}(v_k,\mathrm{Conv}(\Gamma)) .
\end{align*}

So $\mathrm{dist}(v,\mathrm{Conv}(\Gamma)) = \mathrm{dist}(v,x)$.

Step 2: Estimate $x^{n+1}(v)$ and prove $v^{\perp_E}$ is cocausal.

From the definition of $v$ we know:
\[
\langle x,v_k \rangle_E < \langle q_k,v_k \rangle_E \leq \inf_{q \in D_k}\langle q,v_k \rangle_E.
\]

Let $\hat{q_k} = (x^1(q_k),...,x^n(q_k))$, $\hat{v_k} = (x^1(v_k),...,x^n(v_k))$, then
\[
0 < \langle\hat{q_k},\hat{v_k}\rangle_E + \frac{x^{n+1}(v_k)}{k} \leq -(\frac{k-1}{k^2})|\hat{v_k}| + \frac{x^{n+1}(v_k)}{k} .
\]

So $x^{n+1}(v_k) > 0$ and $\frac{|\hat{v_k}|}{x^{n+1}(v_k)} < \frac{k}{k-1}$. We conclude that $x^{n+1}(v) \geq \frac{\sqrt{2}}{2}$. Then $v^{\perp_E}$ is a cocausal hyperplane, $x \in v^{\perp_E}$ and $\mathrm{Conv}(\Gamma)$ is below $v^{\perp_E}$.

\Needspace{5\baselineskip}
Step 3: Construct the roof.

By step 2 we know that there exist two cocausal planes $\Sigma^+$ and $\Sigma^-$ such that:

\[
x \in \Sigma^+ \cap \Sigma^-;
\]

$\mathrm{Conv}(\Gamma)$ is below $\Sigma^+$ and above $\Sigma^-$.

The hyperplanes $\Sigma^+$ and $\Sigma^-$ cannot coincide, since otherwise $\mathrm{Conv}(\Gamma)$ would be contained in their common hyperplane, contrary to the full-dimensional assumption.

If $\Sigma^+ \cap \Sigma^-$ is null, then $\Sigma^+ =\Sigma^-$ are null, and $\Gamma$ has to be a subset of a null hyperplane (which is impossible). So $\Sigma^+ \cap \Sigma^-$ is a spacelike $(n-1)$-subspace, and $(\Sigma^+ \cap \Sigma^-)^\perp$ is a Lorentz $2$-plane.

Let $P_0$ be the orthogonal projection to $(\Sigma^+ \cap \Sigma^-)^\perp$. From the straight observation to a Lorentz $2$-plane, we know there exist two null lines $L^+$ and $L^-$ in $(\Sigma^+ \cap \Sigma^-)^\perp$ such that:

\[
L^+\cap L^- = P_0(\Sigma^+ \cap \Sigma^-);
\]

$P_0(\mathrm{Conv}(\Gamma))$ is below $L^+$ and above $L^-$.

Then $(P_0^{-1}(L^+),P_0^{-1}(L^-))$ is a roof of $\Gamma$ at $x$.
\end{proof}

\begin{corollary}
Let $\Gamma$ be a smooth convex-like submanifold. Then $T_p\Gamma$ is spacelike for each $p\in \Gamma$.
\end{corollary}

\begin{proof}
Let $(\pi^+,\pi^-)$ be a null roof of $\Gamma$ at $p$. It suffices to show that any $v\in T_p\Gamma$ is parallel to $\pi^+\cap\pi^-$.

$\Sigma$ is below $\pi^+$, so $p+v$ is below $\pi^+$ and $p-v$ is below $\pi^+$, hence $v\parallel\pi^+$. Similarly, $v\parallel\pi^-$.
\end{proof}

\begin{proposition}\label{prop-acausal}
Let $\Gamma$ be a convex-like submanifold. Then $\Gamma$ is acausal.
\end{proposition}

\begin{proof}
Suppose, on the contrary, that there exists $p,q\in \Gamma$ such that $p-q$ is causal. Then the projection along $p-q$ maps $p,q$ onto a same point, which contradicts to the definition of convex-like submanifolds.
\end{proof}

In fact, we can obtain a much stronger conclusion than either of the preceding two.

\begin{theorem}\label{uniformly spacelike}
A convex-like submanifold $\Gamma$ is uniformly spacelike, i.e. fix any Lorentzian coordinate,
\[
\sup_{x,y \in \Gamma,x \neq y}\frac{|x-y|_E}{|x-y|_L} < +\infty.
\]
\end{theorem}

\begin{proof}
Assume
\[
\sup_{x,y \in \Gamma,x \neq y}\frac{|x-y|_E}{|x-y|_L} = +\infty.
\]
Suppose a sequence of pairs $\{(x_k,y_k)\}$ satisfies
\[
\frac{|x_k-y_k|_E}{|x_k-y_k|_L} \geq k.
\]
By compactness, we may assume that $x_k \rightarrow p$, $y_k \rightarrow p$ (because there is no null chord on a convex-like submanifold) and
\[
\frac{x_k-y_k}{|x_k-y_k|_E} \rightarrow L \in \mathrm{Nullcone}^+.
\]

Let $\pi_k^+$ be an upper null support at $y_k$, $L_k \in (\pi_k^+)^\perp \cap \mathrm{Nullcone}^+$, $|L_k|_E = 1$(in other words, $L_k$ is the Euclidean-unit future-pointing Lorentzian normal vector to $\pi_k^+$). Again by compactness we assume $L_k \rightarrow L'$. By the definition of an upper null support, we have

\[\langle L_k, x_k-y_k \rangle  \geq 0.\]

Let $k \rightarrow \infty$ then

\[\langle L', L \rangle  \geq 0.\]

So $L' = L$.

Then again by the definition of an upper null support, for $q \in \Gamma$ we have:

\[\langle L_k, q-y_k \rangle  \geq 0.\]

Taking limit yields

\[\langle L, q-p \rangle  \geq 0.\]

So the hyperplane
\[
\pi^+ = \{u\in \mathbb{R}_1^{n+1}|\langle L, u-p \rangle  = 0\}
\]
is an upper null support of $\Gamma$ at $p$.

By the same argument, $\frac{y_k-x_k}{|y_k-x_k|_E} \rightarrow -L \in \mathrm{Nullcone}^-$. So the hyperplane
\[
\pi^- = \{u\in \mathbb{R}_1^{n+1}|\langle -L, u-p \rangle  = 0\}
\]
is a lower null support of $\Gamma$ at $p$. Contradiction!
\end{proof}

\section{Towers and Caps}

In this section, we introduce \emph{towers} as the convex body bounded by null supports and \emph{caps} as the null boundaries of towers. The key to the construction of towers and caps is the support function on $\mathrm{Nullcone}^+$ and $\mathrm{Nullcone}^-$.

\begin{definition}
Given any continuous, positively $1$-homogeneous function $M^+:\mathrm{Nullcone}^+ \rightarrow R$, its corresponding \emph{face-up tower} is defined to be
\[
\Omega^+(M^+) = \{x\in \mathbb{R}_1^{n+1} | \langle x,L \rangle >M^+(L),\forall L \in \mathrm{Nullcone}^+\}.
\]

Similarly, given any continuous, positively $1$-homogeneous function $M^-:\mathrm{Nullcone}^- \rightarrow R$, its corresponding \emph{face-down tower} is
\[
\Omega^-(M^-) = \{x\in \mathbb{R}_1^{n+1} | \langle x,L \rangle >M^-(L),\forall L \in \mathrm{Nullcone}^-\}.
\]

We call $M^\pm$ the \emph{support function} of $\Omega^\pm(M^\pm)$.

\end{definition}

\begin{proposition}\label{tower and timecone}
(a)Every tower $\Omega^\pm(M^\pm)$ is a subset of the timelike cone $\mathrm{Timecone}^\mp (x)$ for some $x\in\mathbb{R}_1^{n+1}$.

(b)If $x \in \Omega^\pm(M^\pm)$, then $\overline{\mathrm{Timecone}^\mp (x)} \subset \Omega^\pm(M^\pm)$.
\end{proposition}

\begin{proof}
(a)Without loss of generality take a face-up tower. $M^+$ is homogeneous and continuous, so there exists $x \in \mathbb{R}_1^{n+1}$ that
\[
\langle x,L^+ \rangle < M(L^+),\forall L^+ \in \mathrm{Nullcone}^+.
\]
This shows $\Omega^+ \subset \mathrm{Timecone}^-(x)$.

(b)If $y \in \overline{\mathrm{Timecone}^\mp (x)}$, then
\[
\langle x,L^\pm \rangle \leq \langle y,L^\pm \rangle,\forall L^\pm \in \mathrm{Nullcone}^\pm,
\]
hence $y \in \Omega^\pm(M^\pm)$.
\end{proof}

\begin{definition}[The corresponding tower]
For a bounded subset $\Gamma\subset \mathbb{R}_1^{n+1}$ we define its corresponding face-up tower
\[
\Omega^+(\Gamma) = \{x\in \mathbb{R}_1^{n+1} | \langle x,L \rangle >\inf_{p\in \Gamma} \langle p,L \rangle ,\forall L \in \mathrm{Nullcone}^+\};
\]
and its corresponding face-down tower
\[
\Omega^-(\Gamma)=\{x\in \mathbb{R}_1^{n+1} | \langle x,L \rangle >\inf_{p\in \Gamma} \langle p,L \rangle ,\forall L \in \mathrm{Nullcone}^-\}.
\]

In other words, the support function of $\Omega^\pm(\Gamma)$ is
\[
M^\pm(L) = \inf_{p\in \Gamma} \langle p,L \rangle
\]
for $L\in \mathrm{Nullcone}^\pm$. Hence we call $M^\pm$ the \emph{upper(lower) support function} of $\Gamma$.
\end{definition}

\begin{proposition}\label{nonempty corresponding towers}
If $\Gamma$ is a convex-like submanifold, then
\[
\Omega^+(\Gamma)\cap\Omega^-(\Gamma)\neq\emptyset.
\]
We denote by
\[
\Omega(\Gamma) = \Omega^+(\Gamma)\cap\Omega^-(\Gamma)
\]
the \emph{acausal interior} of $\Gamma$.
\end{proposition}

\begin{proof}
$\Gamma$ cannot lie in a null hyperplane: projection along its null direction would place $P(\Gamma)$ in an affine hyperplane of the target, contrary to the definition of a convex-like submanifold. Choose $q$ in the relative interior of $\mathrm{Conv}(\Gamma)$. Every nonzero null linear functional is nonconstant on this convex hull, so
\[
\langle q,L\rangle>\inf_{x\in\Gamma}\langle x,L\rangle
\quad\text{for every }L\in\mathrm{Nullcone}^+\cup\mathrm{Nullcone}^-.
\]
Hence $q\in\Omega(\Gamma)$.
\end{proof}

\begin{definition}
Given a support function $M^\pm$ (of a tower $\Omega^\pm$), we define its corresponding \emph{cap}:
\[
Cap^\pm(M) = \partial \Omega^\pm(M^\pm).
\]

Similarly we can define the corresponding cap to a bounded subset $\Gamma\subset \mathbb{R}_1^{n+1}$:
\[
Cap^\pm(\Gamma) = \partial \Omega^\pm(\Gamma).
\]

$M^\pm$ is called the \emph{support function} of $Cap^\pm(M^\pm)$.

If the support function is defined on $\mathrm{Nullcone}^+$($\mathrm{Nullcone}^-$), then its corresponding cap is said to be face-up(face-down). In other words, the boundary of a face-up(face-down) tower is said to be a face-up(face-down) cap.
\end{definition}

\begin{remark} For convenience, if having determined whether a tower or a cap is face-up or face-down or having determined whether a support function is defined on $\mathrm{Nullcone}^+$ or $\mathrm{Nullcone}^-$, we can omit all the plus-minus signs.
\end{remark}

\begin{proposition}
\label{timelike projection of caps}
Let $C$ be a cap in $\mathbb{R}_1^{n+1}$, and let
$P_{v,\Sigma}$ be a timelike projection onto a spacelike
hyperplane $\Sigma$. Then
\[
P_{v,\Sigma}|_C:C\longrightarrow\Sigma
\]
is a homeomorphism.
\end{proposition}

\begin{proof}
It suffices to consider a face-up cap
$C=Cap^+(M^+)$; the face-down case follows by reversing
the time-orientation.
Without loss of generality, assume that $v$ is
future-pointing.

The normalized null section
\[
S_v=\{L\in\mathrm{Nullcone}^+:
          \langle v,L\rangle=-1\}
\]
is compact.
For $y\in\Sigma$, define
\[
f(y)=\min_{L\in S_v}
       \bigl(\langle y,L\rangle-M^+(L)\bigr).
\]
Since $M^+$ is continuous, this minimum exists and is finite.
Moreover, for any $y,z\in\Sigma$,
\[
\begin{aligned}
|f(y)-f(z)|
&\leq \max_{L\in S_v}|\langle y-z,L\rangle|\\
&\leq
\left(\max_{L\in S_v}|L|_E\right)|y-z|_E.
\end{aligned}
\]
Thus $f$ is Lipschitz continuous.

Every point of $\mathbb{R}_1^{n+1}$ has a unique
representation $y+tv$, where $y\in\Sigma$ and
$t\in\mathbb{R}$.
By the positive homogeneity of $M^+$ and the identity
$\langle v,L\rangle=-1$ on $S_v$, we have
\[
y+tv\in\Omega^+(M^+)
\quad\Longleftrightarrow\quad
t<f(y).
\]
Consequently, continuity of $f$ gives
\[
C=\partial\Omega^+(M^+)
 =\{y+f(y)v:y\in\Sigma\}.
\]
Hence $P_{v,\Sigma}|_C$ is bijective, with inverse
\[
y\longmapsto y+f(y)v.
\]
Both maps are continuous, which finishes the proof.
\end{proof}

\begin{example}
(a) Let $v$ be a future-pointing unit timelike vector and $\Gamma$ be the unit sphere in the tangent space $v^\perp$. The face-up tower of $\Gamma$ is exactly $\mathrm{Timecone}^-(v)$ and the face-up cap is $\mathrm{Nullcone}^-(v)\cup \{v\}$.

(b) Let $\Gamma$ be the square

\[\{(x,y,z)\in \mathbb{R}_1^3|max\{|x|,|y|\} = 1 \ , z=0\}.\]

Its upper support function is

\[M^+((\cos\theta,\sin\theta,1)) = -|\cos\theta|-|\sin\theta|.\]

Its corresponding cap consists of four flat null faces and four quarter-Nullcones, see Figure~\ref{square's cap}.

\begin{figure}[htbp]
\centering
\includegraphics[width=0.62\linewidth]{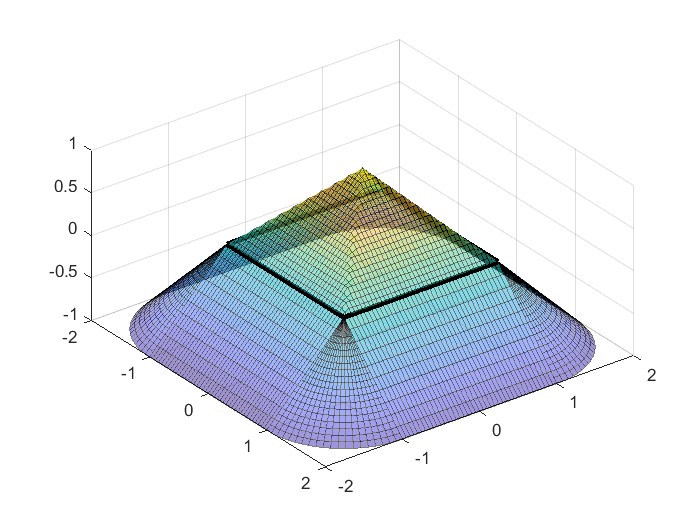}
\caption{The face-up cap of a square}
\label{square's cap}
\end{figure}

(c)Consider the ellipse in $\mathbb{R}_1^3$, given by

\[x^2/4+y^2 = 1;z=0,\]

Its face-up cap is drawn in Figure~\ref{ellipse's cap}. Away from the singular set
\[
\left\{
\left(x,0,\sqrt{1-\frac{x^2}{3}}\right)
\;\middle|\;
-\frac{3}{2}\leq x\leq\frac{3}{2}
\right\},
\]
the cap is a smooth null hypersurface.

\begin{figure}[htbp]
\centering
\includegraphics[width=0.62\linewidth]{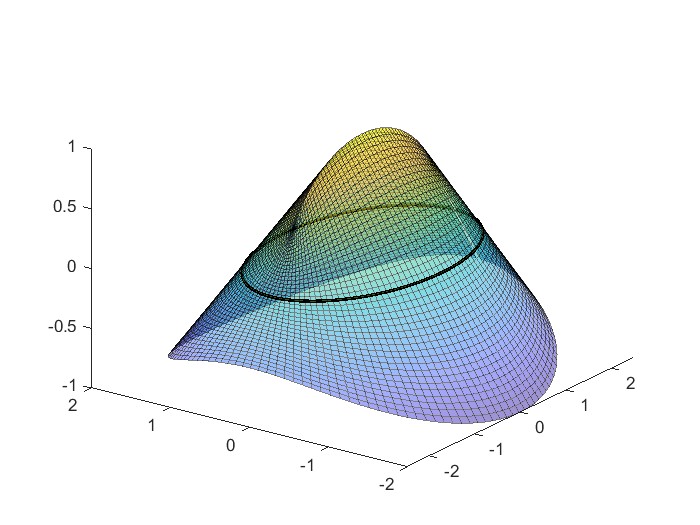}
\caption{The face-up cap of an ellipse}
\label{ellipse's cap}
\end{figure}

\end{example}

\begin{proposition}\label{caps' support representation}
Every face-up tower $\Omega^+(M)$ is an open convex set and its boundary $Cap^+(M)$ is exactly the set:
\[
\{x \in \mathbb{R}_1^{n+1}|\langle x,L \rangle \geq M(L),\forall L \in \mathrm{Nullcone}^+ \text{ and equality holds for some } L\}.
\]
\end{proposition}

\begin{proof}
For convenience let
\[
C = \{x \in \mathbb{R}_1^{n+1}|\langle x,L \rangle \geq M(L),\forall L \in \mathrm{Nullcone}^+ \text{ and equality holds for some } L\}.
\]

The convexity of $\Omega^+(M)$ is natural and we only need to prove $C = \partial \Omega^+(M)$.

If $x\in C$, choose an arbitrary future-pointing timelike vector $v$ then $x+v \notin \Omega^+(M)$ and $x-v \in \Omega^+(M)$. So $x \in \partial \Omega^+(M)$.

If $x \in \Omega^+(M)$, choose an arbitrary future-pointing timelike vector $v$ and consider the set
\[
S = \{L^+ \in \mathrm{Nullcone}^+ | \langle L^+,v \rangle  = -1 \}.
\]
By continuity there exists a positive $\epsilon$ so that
\[
\langle x,L^+ \rangle  > M(L^+) + \epsilon, \forall L^+ \in S.
\]
$M$ is homogeneous so
\[
\Omega^+(M) = \{x\in \mathbb{R}_1^{n+1} | \langle x,L \rangle >M(L),\forall L \in S\}.
\]
$S$ is bounded so $x \in \mathrm{int}(\Omega^+(M))$.

If $x \notin C\cup \Omega^+(M)$, then for some $L^+$ and positive $\epsilon$ $\langle x,L^+ \rangle  < M(L^+) - \epsilon$. So $x \notin \overline{\Omega^+(M)}$.
\end{proof}

\begin{proposition}\label{outward Nulls}
Suppose $x \in Cap^+(M)$, $L \in \mathrm{Nullcone}^+$ and $\langle x,L \rangle  = M(L)$. Then for $t \geq 0$, $x - tL \in Cap^+(M)$ .

\end{proposition}

\begin{proof}
$L$ is null so
\[\langle x-tL,L \rangle  = M(L),\]

and for $L' \in \mathrm{Nullcone}^+$,

\[\langle x-tL,L' \rangle  \geq \langle x,L' \rangle  \geq M(L').\]
\end{proof}

\begin{theorem}\label{intersection of caps}
Given a face-up tower $\Omega^+ = \Omega^+(M^+)$ and a face-down tower $\Omega^- = \Omega^-(M^-)$ , if $\Omega=\Omega^+ \cap \Omega^-$ is non-empty, then $\Gamma = Cap^+(M^+) \cap Cap^-(M^-)$ is a convex-like submanifold.
\end{theorem}

\begin{proof}
By Proposition~\ref{tower and timecone}, $\Omega$ is bounded. Choose an arbitrary timelike or null vector $v$ and consider a projection $P = P_v$. It suffices to prove that $P|_\Gamma$ is a homeomorphism from $\Gamma$ to $\partial P(\overline{\Omega})$.

Step 1: $P(\Gamma)\subseteq\partial P(\overline{\Omega})$.

If $x\in\Gamma$ and $P(x)\in\mathrm{int}(P(\overline{\Omega}))$, then there exists $y\in\Omega$ with $P(y)=P(x)$. Thus $y-x$ is a nonzero causal vector. If it is future-pointing, choose $L^+\in\mathrm{Nullcone}^+$ with $\langle x,L^+\rangle=M^+(L^+)$. Then
\[
\langle y-x,L^+\rangle\leq0,
\]
whereas $y\in\Omega^+$ gives $\langle y-x,L^+\rangle>0$, a contradiction. If $y-x$ is past-pointing, the same argument uses a past-pointing lower support normal.

Step 2: Prove injectivity.

Suppose $x_1,x_2 \in \Gamma$, $x_1\neq x_2$, and $P(x_1) = P(x_2)$. Without loss of generality assume $x_1 - x_2$ is a future-pointing timelike or null vector.

From Proposition~\ref{caps' support representation}, for some $L_2 \in \mathrm{Nullcone}^+$, $\langle x_2,L_2 \rangle  = M^+(L_2)$ and thus $\langle x_1 - x_2,L_2 \rangle  \geq 0$, so $x_1 - x_2 \parallel L_2$.

Also we deduce that for some $L_1 \in \mathrm{Nullcone}^-$ $\langle x_1,L_1 \rangle  = M^-(L_1)$ and  $x_1 - x_2 \parallel L_1$.

Then the common null support hyperplane through $x_1$ and $x_2$ separates the open towers $\Omega^+$ and $\Omega^-$, contradicting $\Omega^+\cap\Omega^-\neq\emptyset$!

Step 3: Prove surjectivity.

Consider any point $y \in \partial P(\overline{\Omega})$, then there exists some $z \in \partial{\Omega}$ such that $P(z) = y$. It suffices to find $x \in \Gamma$ such that $P(x) = y$.

The case that $z\in\Gamma$ is trivial. Otherwise, without loss of generality suppose $z \in Cap^+ \cap \Omega^-$ and we may assume the projection vector $v$ is past-pointing. Then for $t \geq 0$ , $z+tv \in \overline{\Omega^+}$.

Next we show that $\forall t,z+tv \in Cap^+(M^+)$. If for some $t \geq 0$, $z+tv \in \Omega^+$, then convexity shows $z+\epsilon tv \in \Omega^+, \forall \epsilon \in (0,1)$. Then for some small $\epsilon$, $z+\epsilon tv \in \Omega^+ \cap \Omega^-$, which contradicts to the fact that $y = P(z+\epsilon tv)\in \partial P(\overline{\Omega})$.

Finally, by Proposition~\ref{tower and timecone}, $\Omega^-$ is contained in $\mathrm{Timecone}^+(a)$ for some point $a$, so the past-pointing ray $z+tv$ eventually leaves $\Omega^-$. By continuity there exists $t_0 > 0$ such that $z+t_0v \in Cap^+(M^+) \cap Cap^-(M^-)$. Then $z+t_0v \in \Gamma$ and $P(z+t_0v) = y$.

Step 4: Notice that $\Gamma$ and $\partial P(\overline{\Omega})$ is compact, so the inverse of $P|_\Gamma$ is continuous. Then $P|_\Gamma$ is a homeomorphism from $\Gamma$ to $\partial P(\overline{\Omega})$.
\end{proof}

\begin{theorem}\label{intersection of corresponding caps}
If $\Gamma$ is a convex-like submanifold then
\[
\Gamma = Cap^+(\Gamma) \cap Cap^-(\Gamma).
\]
\end{theorem}

\begin{proof}
We first show that $\Gamma \subset Cap^+(\Gamma) \cap Cap^-(\Gamma)$: For arbitrary $x \in \Gamma$, from Theorem~\ref{existence of roof} there exists a roof $(\pi^+, \pi^-)$ of $\Gamma$ at $x$. Let $L^+$ be a future-pointing null vector which satisfies $\pi^+ \perp L^+$. Then $\langle x,L^+ \rangle  = \inf_{y \in \Gamma} \langle y,L^+ \rangle $ so $x \in Cap^+(\Gamma)$. By the same argument $x \in Cap^-(\Gamma)$.

By Proposition~\ref{nonempty corresponding towers}, $\Omega(\Gamma)\neq\emptyset$. Thus, by Theorem~\ref{intersection of caps}, the intersection $\widetilde\Gamma = Cap^+(\Gamma) \cap Cap^-(\Gamma)$ is a convex-like submanifold. Thus $\Gamma$ and $\widetilde\Gamma$ are compact manifolds without boundary of the same dimension. By invariance of domain, the inclusion $\Gamma\hookrightarrow\widetilde\Gamma$ has open image. Its image is also closed by compactness. Since $\widetilde\Gamma$ is connected, we conclude that $\Gamma = Cap^+(\Gamma) \cap Cap^-(\Gamma)$.
\end{proof}

Due to convexity and degeneracy, caps are always easier to handle than convex-like submanifolds. Theorem~\ref{intersection of caps} and Theorem~\ref{intersection of corresponding caps} enable us to construct and perturb convex-like submanifolds by working on caps.

An equivalent definition of convex-like submanifolds is provided by the following corollary.

\begin{corollary}\label{caps make convex-like}
Suppose $\Gamma$ is a subset of $\mathbb{R}_1^{n+1}$. Then the following statements are equivalent.
\begin{enumerate}
    \item $\Gamma$ is a convex-like submanifold.
    \item $\exists$ towers $\Omega^+(M^+)$, $\Omega^-(M^-)$,$\Omega^+(M^+)\cap\Omega^-(M^-) \neq \emptyset$ and $\Gamma = Cap(M^+) \cap Cap(M^-)$ (In other words, $\Gamma$ is the intersection of a pair of caps corresponding to a pair of intersecting towers).
\end{enumerate}
\end{corollary}

Generically, there's no uniqueness when choosing two caps to intersect at a given convex-like submanifold. \begin{example}

Consider the square
\[
\{(x,y,z)\in \mathbb{R}_1^3|max\{|x|,|y|\} = 1 \ , z=0\}.
\]

Its corresponding face-up cap is $Cap^+(M)$ where
\[
M((\cos\theta,\sin\theta,1)) = -|\cos\theta|-|\sin\theta|.
\]

Consider another support function:
\[
M_1((\cos\theta,\sin\theta,1)) = 1-2|\cos\theta|-2|\sin\theta|,
\]
extended homogeneously to $\mathrm{Nullcone}^+$;
then $Cap^+(M_1)$ and $Cap^+(M)$ coincide at $z\geq0$, where both have height $z=1-\max\{|x|,|y|\}$. They differ outside this part: at $(x,y)=(2,2)$ their heights are $-1$ and $-\sqrt2$, respectively.

\begin{figure}[htbp]
\centering
\begin{minipage}{0.48\linewidth}
\centering
    \includegraphics[width=\linewidth]{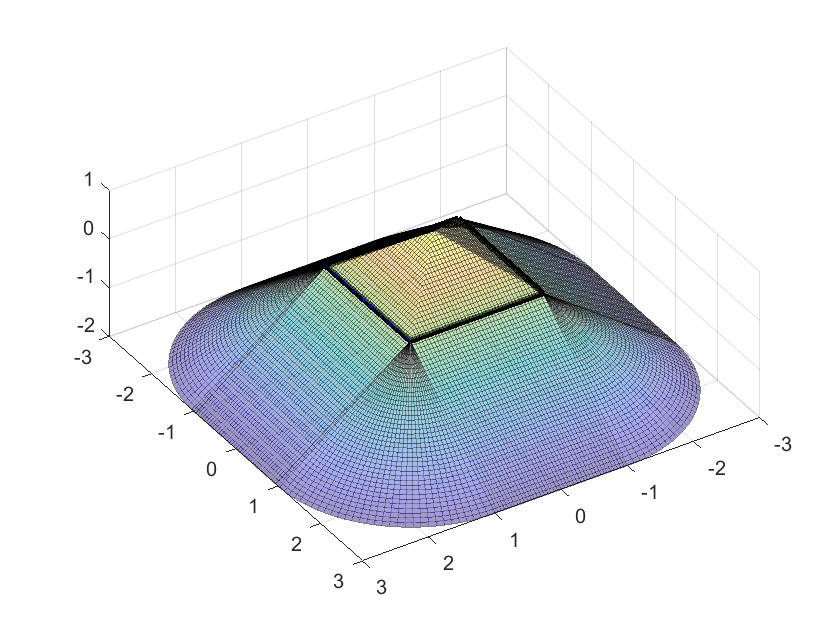}
\end{minipage}\hfill
\begin{minipage}{0.48\linewidth}
\centering
    \includegraphics[width=\linewidth]{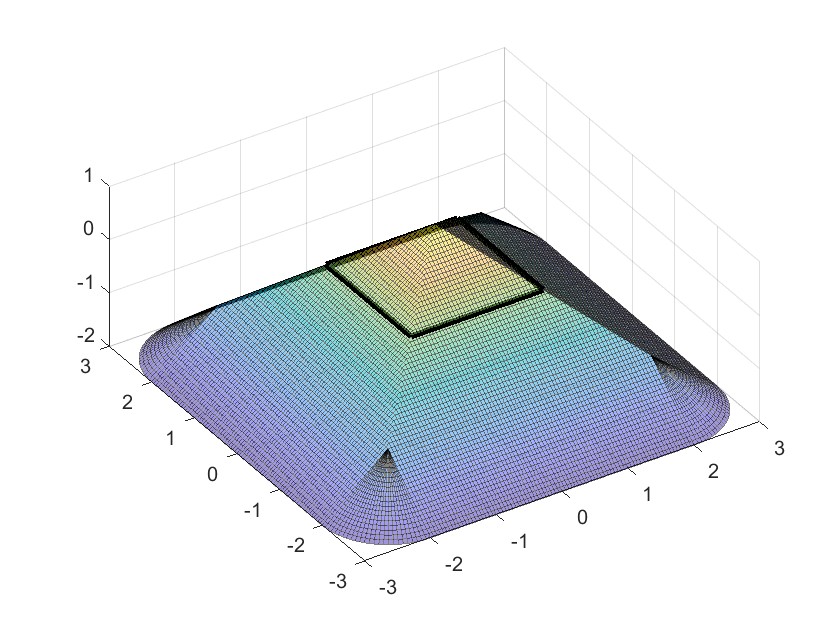}
\end{minipage}
\caption{$Cap^+(M)$(left) and $Cap^+(M_1)$(right)}
\label{square caps comparison}
\end{figure}

\end{example}

In order to obtain a uniqueness result, we may divide a cap into two parts with a given convex-like submanifold as their common boundary.


\begin{definition}
Suppose we are given a cap $\Sigma$ and a convex-like submanifold $\Gamma\subset\Sigma$. By the Jordan-Brouwer Separation Theorem, $\Sigma\setminus\Gamma$ has a bounded and an unbounded connected component. Their closures in $\Sigma$ are called, respectively, the \emph{inner part} and the \emph{outer part} of $\Sigma$ with respect to $\Gamma$. Both parts include their common boundary $\Gamma$.
\end{definition}

\begin{proposition}
\label{projection of inner and outer parts}
Let $C$ be a cap in $\mathbb{R}_1^{n+1}$, $n\geq2$, and
let $\Gamma\subset C$ be a convex-like submanifold.
Denote by $C_{\mathrm{in}}$ and $C_{\mathrm{out}}$ the inner
and outer parts of $C$ with respect to $\Gamma$, respectively.

For any timelike projection $P=P_{v,\Sigma}$ onto a spacelike
hyperplane $\Sigma$, let
\[
K=\mathrm{Conv}(P(\Gamma)).
\]
Then $K$ is a compact convex body in $\Sigma$, and
\[
P(C_{\mathrm{in}})=K,
\qquad
P(C_{\mathrm{out}})
=\Sigma\setminus\operatorname{int}_{\Sigma}K.
\]
Both restrictions are homeomorphisms onto their images.
Moreover,
\[
P(C_{\mathrm{in}}\setminus\Gamma)
=\operatorname{int}_{\Sigma}K,
\qquad
P(C_{\mathrm{out}}\setminus\Gamma)
=\Sigma\setminus K.
\]
\end{proposition}

\begin{proof}
By Proposition~\ref{timelike projection of caps},
$P|_C:C\to\Sigma$ is a homeomorphism.
Since $\Gamma$ is convex-like, $P(\Gamma)=\partial K$.
Thus $P|_C$ maps the connected components of
$C\setminus\Gamma$ onto the two connected components
\[
\operatorname{int}_{\Sigma}K
\quad\text{and}\quad
\Sigma\setminus K
\]
of $\Sigma\setminus\partial K$.

The inverse image of $K$ under $P|_C$ is compact,
because $K$ is compact and $(P|_C)^{-1}$ is continuous.
Consequently, the component corresponding to
$\operatorname{int}_{\Sigma}K$ is bounded.
The component corresponding to $\Sigma\setminus K$
is unbounded, since its image under the affine map $P$
is unbounded.

Therefore these are precisely the bounded and unbounded
components used to define $C_{\mathrm{in}}$ and
$C_{\mathrm{out}}$.
Taking their closures and using the homeomorphism $P|_C$
proves all the assertions.
\end{proof}

\begin{theorem}\label{inward uniqueness of caps}
Let $\Sigma^+$ be a face-up cap and $\Gamma\subset \Sigma^+$ be a convex-like submanifold. Suppose $A$ is the inner part of $\Sigma^+$ with respect to $\Gamma$ and $B$ is the inner part of $Cap^+(\Gamma)$ with respect to $\Gamma$. Then $A=B$.
\end{theorem}

\begin{proof}
It suffices to show that $A \subset B$, since, by Proposition \ref{timelike projection of caps} and Proposition \ref{projection of inner and outer parts}, $A$ and $B$ are both graphs of the (time coordinate component) height function on $P(\mathrm{Conv}(\Gamma))$ for some timelike projection $P$.

Suppose $x \in A$. It suffices to show:

(a)$\forall L^+ \in \mathrm{Nullcone}^+$, $\langle x,L^+ \rangle \geq \inf_{p\in \Gamma} \langle p,L^+ \rangle $;

(b)$\exists L^+ \in \mathrm{Nullcone}^+$, $\langle x,L^+ \rangle = \inf_{p\in \Gamma} \langle p,L^+ \rangle $.

$x$ is in the inner part of $\Sigma^+$ with respect to $\Gamma$, so by continuity, under any null projection $P$, $P(x) \in \mathrm{Conv}(P(\Gamma))$. Hence $x$ is below every upper null support of $\Gamma$ so (a) holds.

By Proposition~\ref{caps' support representation}, there exists $L^+\in\mathrm{Nullcone}^+$ such that
\[
\langle x,L^+ \rangle = \inf_{p\in \Sigma^+} \langle p,L^+ \rangle.
\]

By Proposition~\ref{outward Nulls}, $x-tL^+\in\Sigma^+$ for $t\geq0$. If $x\in\Gamma$, take $t=0$. Otherwise, since $A$ is bounded, this ray first meets its boundary $\Gamma$ at some $t>0$. Thus in either case there exists $t\geq0$ such that $x-tL^+\in\Gamma$. Hence
\[
\inf_{p\in \Gamma} \langle p,L^+ \rangle\leq \langle x-tL^+,L^+ \rangle = \langle x,L^+ \rangle = \inf_{p\in \Sigma^+} \langle p,L^+ \rangle \leq \inf_{p\in \Gamma} \langle p,L^+ \rangle,
\]
so (b) holds.
\end{proof}

\Needspace{4\baselineskip}
Moreover, under a $C^1$ assumption we can acquire the complete uniqueness theorem.

\begin{theorem}\label{outward uniqueness of caps}
Let $C$ be the outer part of $\Sigma^+$ and $D$ be the outer part of $Cap^+(\Gamma)$ with respect to $\Gamma$. If $\Gamma$ is $C^1$ then $C=D$.
\end{theorem}

\begin{proof}
Fix an orthogonal timelike projection $P_v$ onto a spacelike hyperplane $\pi$, and put $K=\mathrm{int}\,\mathrm{Conv}(P_v(\Gamma))$. By Proposition~\ref{outward Nulls}, for every $x\in\Gamma$ there are future-pointing null vectors $L_x,L'_x\perp T_x\Gamma$ such that
\[
x-tL_x\in\Sigma^+,
\qquad x-tL'_x\in Cap^+(\Gamma),
\qquad t\geq0.
\]
Both generating rays point outward relative to $\Gamma$. Indeed, the projection of either ray cannot enter $K$: since each ray is unbounded, it would then meet $\partial K=P_v(\Gamma)$ again, giving a second point of $\Gamma$ on the same null line. This contradicts the injectivity of null projections on $\Gamma$.

Let $\nu_x\in\pi$ be the unit outward normal to $\partial K$ at $P_v(x)$, and let $\tau_x$ be the future-pointing unit timelike vector in $(T_x\Gamma)^\perp$ orthogonal to $\nu_x$. The two future-pointing null normal rays are generated by $\tau_x+\nu_x$ and $\tau_x-\nu_x$. Since the generating direction is $-L_x$ (respectively $-L'_x$), its outward direction selects $\tau_x-\nu_x$ in both cases. Thus, after a positive rescaling, we may take
\[
L_x=L'_x=\tau_x-\nu_x.
\]
These vectors depend continuously on $x$, because $\Gamma$ is $C^1$.

We want to show that $\{x-tL_x\mid x\in\Gamma,\ t\geq0\}=C$. Let $K'\subset\pi$ be any bounded open convex body with $\overline K\subset K'$. Define $f:\partial K\to\partial K'$ by
\[
f(P_v(x))=\{P_v(x-tL_x)\mid t\geq0\}\cap\partial K'.
\]
Since $P_v(x)\in K'$, the ray has exactly one intersection with $\partial K'$, and $f$ is continuous. Fix $Q\in K$. At $z=P_v(x)$, interpolate the outward ray direction $-P_v(L_x)$ with $z-Q$. Both have positive inner product with $\nu_x$, so this interpolation gives a homotopy from $f$ to the central projection from $Q$ onto $\partial K'$. Hence $f$ has degree $\pm1$ and is surjective.

Every point of $\pi\setminus\overline K$ lies on the boundary of some such $K'$, while $\partial K$ is obtained with $t=0$. Consequently,
\[
\{P_v(x-tL_x)\mid x\in\Gamma,\ t\geq0\}=\pi\setminus K.
\]
Since a timelike projection maps each cap homeomorphically onto $\pi$, this proves $\{x-tL_x\mid x\in\Gamma,\ t\geq0\}=C$. The same argument gives equality with $D$, and therefore $C=D$.
\end{proof}

\section{Nesting}

\begin{definition}
Let $\Gamma$ be a convex-like submanifold and $p$ be a point. We say $p$ \emph{nests in} $\Gamma$ if, for every causal projection $P$,
\[
P(p)\in\mathrm{Conv}(P(\Gamma)).
\]
The convex hull here includes its boundary, so contact is allowed.
\end{definition}

\begin{proposition}
Let $\Gamma$ be a convex-like submanifold. The following statements are equivalent:

(a)$p$ nests in $\Gamma$;

(b)$p\in\overline{\Omega(\Gamma)}$.

Thus the set of all points that nest in $\Gamma$ is $\overline{\Omega(\Gamma)}$.
\end{proposition}

\begin{proof}
By Proposition~\ref{nonempty corresponding towers}, the two corresponding towers have a common interior point. Since they are convex, their intersection satisfies the closure identity in (b).

(a)$\to$(b): If $p\notin\overline{\Omega^+(\Gamma)}$, then for some future-pointing null vector $L$,
\[
\langle p,L\rangle<\inf_{x\in\Gamma}\langle x,L\rangle.
\]
For a projection $P$ along $L$, the functional $\langle\,\cdot\,,L\rangle$ is unchanged by projection. The displayed inequality therefore implies $P(p)\notin\mathrm{Conv}(P(\Gamma))$, a contradiction. The same argument with past-pointing null vectors gives $p\in\overline{\Omega^-(\Gamma)}$.

(b)$\to$(a): By Theorem~\ref{intersection of corresponding caps} and the proof of Theorem~\ref{intersection of caps}, for every causal projection $P$,
\[
P\bigl(\overline{\Omega(\Gamma)}\bigr)=\mathrm{Conv}(P(\Gamma)).
\]
Thus $p\in\overline{\Omega(\Gamma)}$ implies $P(p)\in\mathrm{Conv}(P(\Gamma))$ for every such $P$.
\end{proof}

\begin{definition}
Let $\Gamma_1$ be any subset of $\mathbb{R}_1^{n+1}$ and $\Gamma_2$ be a convex-like submanifold. We say $\Gamma_1$ \emph{nests in} $\Gamma_2$ if every point of $\Gamma_1$ nests in $\Gamma_2$, or equivalently, if for every causal projection $P$,
\[
P(\Gamma_1)\subseteq\mathrm{Conv}(P(\Gamma_2)).
\]
In particular, contact and equality are allowed.
\end{definition}

\begin{proposition}
Let $\Gamma_1$ and $\Gamma_2$ be two convex-like submanifolds. Then the following statements are equivalent:

$($a$)$$\Gamma_1$ nests in $\Gamma_2$;

$($b$)$$\Omega(\Gamma_1)\subseteq\Omega(\Gamma_2)$.
\end{proposition}

\begin{proof}
Suppose (a) holds. By the pointwise characterization above, $\Gamma_1\subseteq\overline{\Omega(\Gamma_2)}$. Writing
\[
M_i^\pm(L)=\inf_{x\in\Gamma_i}\langle x,L\rangle,
\qquad L\in\mathrm{Nullcone}^\pm,
\]
we obtain $M_1^\pm(L)\geq M_2^\pm(L)$. Therefore
\[
\Omega^\pm(\Gamma_1)\subseteq\Omega^\pm(\Gamma_2),
\]
which implies (b).

Conversely, (b) gives
\[
\Gamma_1\subseteq\overline{\Omega(\Gamma_1)}
\subseteq\overline{\Omega(\Gamma_2)},
\]
so every point of $\Gamma_1$ nests in $\Gamma_2$. Moreover, each $\Omega(\Gamma_i)$ is the interior of its closure. Thus (b) is also equivalent to
\[
\overline{\Omega(\Gamma_1)}\subseteq\overline{\Omega(\Gamma_2)}.
\]

\end{proof}

\begin{definition}
Let $\Sigma$ be a cap and $\Gamma_1,\Gamma_2\subset\Sigma$ be two convex-like submanifolds. We say $\Gamma_1$ is \emph{inside $\Gamma_2$ on $\Sigma$} if $\Gamma_1$ is contained in the inner part of $\Sigma$ with respect to $\Gamma_2$. Since the inner part includes its boundary, contact and equality are allowed.
\end{definition}

\begin{proposition}\label{insert a convex-like submanifold}
Let $\Gamma_1$ and $\Gamma_2$ be two convex-like submanifolds. If $\Gamma_1$ nests in $\Gamma_2$, then there exists a face-up cap $\Sigma_1$, a face-down cap $\Sigma_2$ and a convex-like submanifold $\Gamma$, such that:

(a)$\Gamma_1$ is inside $\Gamma$ on $\Sigma_1$;

(b)$\Gamma$ is inside $\Gamma_2$ on $\Sigma_2$.
\end{proposition}

\begin{proof}
Let $\Sigma_1=Cap^+(\Gamma_1)$, $\Sigma_2=Cap^-(\Gamma_2)$ and $\Gamma=\Sigma_1\cap\Sigma_2$. The tower inclusions established above give
\[
\Omega(\Gamma_1)\subseteq D:=\Omega^+(\Gamma_1)\cap\Omega^-(\Gamma_2)
\subseteq\Omega(\Gamma_2).
\]
In particular, $D$ is nonempty, so $\Gamma$ is a convex-like submanifold by Theorem~\ref{intersection of caps}.

Fix a timelike projection $P$. The proof of that theorem identifies $P(\overline D)$ with $\mathrm{Conv}(P(\Gamma))$ and $P(\overline{\Omega(\Gamma_i)})$ with $\mathrm{Conv}(P(\Gamma_i))$. Hence
\[
\mathrm{Conv}(P(\Gamma_1))\subseteq\mathrm{Conv}(P(\Gamma))
\subseteq\mathrm{Conv}(P(\Gamma_2)).
\]
On each cap, $P$ maps the inner part with respect to a convex-like submanifold onto the closed convex body bounded by its projection. Since $\Gamma_1,\Gamma\subset\Sigma_1$ and $\Gamma,\Gamma_2\subset\Sigma_2$, the two displayed inclusions give (a) and (b), respectively.
\end{proof}

\section{Volume comparison}

In the Euclidean space $\mathbb{R}^n$, there holds a volume comparison theorem:

\begin{theorem}\label{euclidean volume comparison}
Let $\Gamma_1$ and $\Gamma_2$ be two convex hypersurfaces in $\mathbb{R}^n$. If $\Gamma_1$ is enclosed by $\Gamma_2$, then
\[
V^{n-1}(\Gamma_1) \leq V^{n-1}(\Gamma_2).
\]
\end{theorem}

There are two ways to prove Theorem~\ref{euclidean volume comparison}.
\begin{proof}[Crofton Formula]
By the Crofton formula \cite[2.10.15 and 3.2.26]{Federer}, we have
\[
V^{n-1}(\Gamma_1) = \frac{1}{\beta(n, n-1)} \int_{P \in \mathbf{O}^*(n, n-1)} \int_{y \in \text { image } P \cong \mathbf{R}^{n-1}} N(\Gamma_1\cap P^{-1}(y))\,d\mathcal{L}^{n-1}(y)\,dP.
\]

$\Gamma_1$ is enclosed by $\Gamma_2$ so
\[
N(\Gamma_1\cap P^{-1}(y))\leq N(\Gamma_2\cap P^{-1}(y))
\]
holds for every $P \in \mathbf{O}^*(n, n-1)$ and almost every $y \in \text { image } P$. We then conclude
\[
V^{n-1}(\Gamma_1) \leq V^{n-1}(\Gamma_2).
\]
\end{proof}

\begin{proof}[Constructing a map]
Let $K_1\subseteq K_2$ be the compact convex bodies bounded by $\Gamma_1$ and $\Gamma_2$, respectively, and let $\pi_{K_1}:R^n\to K_1$ be the Euclidean nearest-point projection. This map is $1$-Lipschitz: if $p=\pi_{K_1}(x)$ and $q=\pi_{K_1}(y)$, the minimizing property gives
\[
\langle x-p,q-p\rangle_E\leq0,
\qquad
\langle y-q,p-q\rangle_E\leq0,
\]
and hence
\[
|p-q|_E^2\leq\langle x-y,p-q\rangle_E
\leq |x-y|_E\,|p-q|_E.
\]
Since $\partial K_2$ does not meet $\operatorname{int}K_1$, the restriction $F=\pi_{K_1}|_{\Gamma_2}$ takes values in $\Gamma_1$.

This restriction is surjective. Indeed, for any $p\in\Gamma_1$, choose an outward Euclidean unit supporting normal $\nu$ to $K_1$ at $p$. The ray $p+t\nu$, $t\geq0$, has a last point $y$ in $K_2$, and $y\in\Gamma_2$. For every $z\in K_1$, the supporting inequality $\langle\nu,z-p\rangle_E\leq0$ shows that $|y-z|_E^2\geq|y-p|_E^2$, so $F(y)=p$. This also includes the case $y=p$.

The boundaries of convex bodies are rectifiable. By the area formula and the bound $J_{n-1}F\leq1$ almost everywhere, we obtain
\[
V^{n-1}(\Gamma_1)
\leq \int_{\Gamma_1}\#F^{-1}(p)\,d\mathcal H^{n-1}(p)
=\int_{\Gamma_2}J_{n-1}F\,d\mathcal H^{n-1}
\leq V^{n-1}(\Gamma_2).
\]
\end{proof}

\begin{remark} The argument based on the Crofton formula does not directly extend to Lorentz space. On the one hand, the set of timelike projections does not admit a finite measure that is invariant under Lorentz transformations. On the other hand, let us consider the following spacelike curves:
\[
r_1(s) = (10\cos s, 10\sin s, 0),\,r_2(s) = (12\cos s, 12\sin s, \frac{\sin 120s}{12}).
\]

The length of $r_1$ is $20\pi$, and the length of $r_2$ is
\[
\begin{aligned}
\int_0^{2\pi}\sqrt{144-100\cos^2 s}\,ds
&\leq \int_0^{2\pi}(12-\frac{100\cos^2 s}{24})\,ds\\
&=24\pi-\frac{25\pi}{6}<20\pi.
\end{aligned}
\]

However, $r_1\subset\mathrm{int}(\mathrm{Conv}(r_2))$, so for any plane $\pi$, if $\pi$ intersects $r_1$ then $\pi$ intersects $r_2$ in at least two points. Then for almost every plane $\pi$, $N(r_2\cap \pi)\geq N(r_1\cap \pi)$. This shows that, generally in Lorentz space, comparing the number of intersection points with each plane cannot derive the volume comparison.
\end{remark}

\begin{remark} This nearest-point argument does not directly extend to Lorentz space. The Lorentzian quadratic form is indefinite, consequently the nearest-point projection map is not $1$-Lipschitz. However, in Section 7 we will construct a cap-restricted map similar to the nearest-point map. Its preimage set is a subset of the outer submanifold, and it is indeed a Lipschitz surjective map. Moreover, this map does not increase the distance between any two points, providing a proof of the volume comparison theorem in the Lorentz space.
\end{remark}

\section{The proof of volume comparison theorem}

In this section, we use the theorems about Lipschitz maps to define the volume of an arbitrary convex-like submanifold. To consider Lipschitz maps in $\mathbb{R}_1^{n+1}$, we take the Euclidean distance as the metric function. Notice that under different choices of coordinates, the Euclidean distances are equivalent.

Let $\mathcal{H}^m$ be the $m$-dimensional Hausdorff measure and $\mathcal{L}^m$ be the $m$-dimensional Lebesgue measure. Throughout this section, a Lipschitz homeomorphism means a bi-Lipschitz homeomorphism, i.e. both the map and its inverse are Lipschitz.

If a map $f:M\to N$ is differentiable at $x$, denote by $Df_x$ the differential of $f$ at $x$. Moreover, if $k=\mathrm{dim}(M) \leq \mathrm{dim}(N)$ and $i_1:T_xM\to R^k$ is isometry, $i_2:Df_x(T_xM)\to R^k$ is linear isometric embedding, then let $Jf(x)$ be the \textbf{absolute value} of the Jacobian determinant of $i_2\circ Df_x\circ i_1^{-1}$.

\begin{proposition}\label{convex-like lip to sphere}
Suppose $\Gamma$ is a convex-like submanifold. Then there exists a homeomorphism $f:\Gamma\to S^{n-1}$ such that $f,f^{-1}$ are Lipschitz.
\end{proposition}

\begin{proof}
Fix an arbitrary coordinate and take the orthogonal timelike projection along $v = (0,...,0,1)$. Its restriction $P : \Gamma \rightarrow \Gamma_0$ from $\Gamma$ to a convex hypersurface $\Gamma_0 \subset v^\perp \cong R^n$. Without loss of generality, let $\Gamma_0\in R^n$ and the unit sphere $S^{n-1}$ be in its interior.

We first show that $P,P^{-1}$ are Lipschitz. Let $x,y\in\Gamma$, we have
\[
|x-y|_E \leq \sqrt{2}|P(x)-P(y)|_E \leq \sqrt{2}|x-y|_E.
\]

Let $K$ be the convex body bounded by $\Gamma_0$, and choose $R>1$ such that $B_1\subset K\subset B_R$. Since the origin is in the interior of $K$, the central projection $Q:\Gamma_0\to S^{n-1}$ is a homeomorphism. By Lemma~\ref{central projection lip}, $Q$ is Lipschitz.

To estimate its inverse, write $x=\alpha u$ and $y=\beta v$, where $x,y\in\Gamma_0$, $u,v\in S^{n-1}$ and $1\leq\beta\leq\alpha\leq R$. Choose a unit outward normal $\nu$ to a supporting hyperplane of $K$ at $y$. Then $\langle\nu,y\rangle_E\geq1$ and $\langle\nu,x\rangle_E\leq\langle\nu,y\rangle_E$. Hence
\[
(\alpha-\beta)\langle\nu,v\rangle_E\leq\alpha\,\langle\nu,v-u\rangle_E,
\qquad \langle\nu,v\rangle_E\geq\frac1\beta,
\]
and consequently
\[
|\alpha-\beta|\leq\alpha\beta|u-v|\leq R^2|u-v|.
\]
By the cosine law,
\[
|x-y|^2=(\alpha-\beta)^2+\alpha\beta|u-v|^2
\leq(R^4+R^2)|u-v|^2.
\]
Thus $Q^{-1}$ is Lipschitz. We obtain the desired map $f=Q\circ P$.
\end{proof}

\begin{definition}
Suppose $\Gamma$ is a convex-like submanifold, $f : \Gamma \rightarrow S^{n-1}$ is a Lipschitz homeomorphism. We define the volume of $\Gamma$ by
\[
V^{n-1}(\Gamma) = \int_{S^{n-1}} J_{n-1}f^{-1}(x) d\mathcal{H}^{n-1}(x)
\]
where $\mathcal{H}^{n-1}$ denotes the $(n-1)$-dimensional Hausdorff measure and $J_{n-1}f^{-1}(x)$ denotes the Lorentzian Jacobian (i.e. the tangent space of $\Gamma$ are equipped with the induced Lorentzian metric.) of $f^{-1}$ at $x$ where $Df^{-1}$ is well defined.

By Corollary~\ref{lip differentiable a.e. on sphere}, $f^{-1}$ is differentiable $\mathcal{H}^{n-1}$ a.e.. By Theorem~\ref{ap derivative} the $Df^{-1}$ is measurable, so the integral is well-defined.

If $f_1$ and $f_2$ are two Lipschitz homeomorphisms, let $F = f_2\circ f_1^{-1}$ and $y = F(x)$. Use Corollary~\ref{lip area formula sphere} and let $u(x) = J_{n-1}f_2^{-1}(y)$, then
\begin{align*}
    &\int_{S^{n-1}} J_{n-1}f_1^{-1}(x) d\mathcal{H}^{n-1}(x)\\
    &= \int_{S^{n-1}} J_{n-1}f_2^{-1}(y)\cdot J_{n-1}F(x) d\mathcal{H}^{n-1}(x)\\
    &=\int_{S^{n-1}} J_{n-1}f_2^{-1}(y) d\mathcal{H}^{n-1}(y),
\end{align*}
which shows that $V^{n-1}(\Gamma)$ is well-defined.
\end{definition}

For smooth convex-like submanifolds, $V^{n-1}(\Gamma)$ coincides with the volume induced by its inherited Riemannian metric.

\begin{lemma}\label{injection enlarging distance}
Let $\Sigma$ be a cap and $\Gamma_1,\Gamma_2\subset \Sigma$ be two convex-like submanifolds. If $\Gamma_1$ is inside $\Gamma_2$ on $\Sigma$, then there exists an injection $\phi : \Gamma_1 \rightarrow \Gamma_2$, satisfying
\[
|\phi(x) - \phi(y)|_L \geq |x-y|_L,\forall x,y \in \Gamma_1.
\]
\end{lemma}

\begin{proof}
Without loss of generality, suppose $\Sigma$ is a face-up cap. By Propositions~\ref{caps' support representation}, for each $x\in\Gamma_1$, there exists $L_x\in\mathrm{Nullcone}^+$ such that
\[
\langle x,L_x\rangle = \inf_{p\in \Gamma_1} \langle p,L_x\rangle.
\]
By ~\ref{outward Nulls}, $x-tL_x\in\Sigma$ for all $t\geq0$.

We construct a map $\phi:\Gamma_1\to\Gamma_2$ as follows.

If $x\in\Gamma_1\cap\Gamma_2$, set $t_x=0$ and $\phi(x)=x$.

Otherwise, $x$ lies in the inner part of $\Sigma$ with respect to $\Gamma_2$. Since this inner part is bounded, the ray $x-tL_x$ first meets its boundary $\Gamma_2$ at some $t_x>0$; set $\phi(x)=x-t_xL_x$. The intersection is unique, because two distinct points of $\Gamma_2$ cannot be joined by a null chord. Thus $t_x\geq0$ in all cases.

Then we calculate
\begin{align*}
|\phi(x) - \phi(y)|_L^2
    &=\langle x-t_xL_x-y+t_yL_y,x-t_xL_x-y+t_yL_y \rangle\\
    &=\langle x-y,x-y \rangle
      +2\bigl(\langle y-x,t_xL_x \rangle
      +\langle x-y,t_yL_y \rangle\\
    &\hspace{6em}{}-\langle t_xL_x,t_yL_y \rangle\bigr)\\
    &\geq\langle x-y,x-y \rangle=|x-y|_L^2
\end{align*}
Since distinct points of $\Gamma_1$ have positive Lorentzian separation, this inequality also proves that $\phi$ is injective. The calculation remains valid when either parameter is zero.
\end{proof}

\begin{remark}
If $\Gamma_1$ and $\Gamma_2$ are $C^2$, then the mapping $x\mapsto T_x\Gamma_1$ (from $\Gamma$ to the Grassman manifold $G_{n-1}(R^{n+1})$) is $C^1$, hence the mapping $x \mapsto L_x$ could be chosen $C^1$. So $\phi$ is $C^1$ and we easily conclude that $V^{n-1}(\Gamma_1)\leq V^{n-1}(\Gamma_2)$. This method works for piecewise $C^2$ submanifolds, so the volume comparison between smooth convex-like submanifolds or convex-like polyhedrons on one cap could be established.

However, if one smooth convex-like submanifold nests in another, we may fail to use Lemma~\ref{injection enlarging distance} twice and Proposition~\ref{insert a convex-like submanifold} to conclude the volume comparison relation. This is because their corresponding caps have singularities, and the intersection of two caps may be no longer $C^1$. So we need to handle the volume of generic convex-like submanifolds.

In fact, if in Lemma~\ref{injection enlarging distance} $\Gamma_1$ is bad enough, then the map $\phi$ may have dense discontinuity.But the volume comparison could still be deduced from a Lipschitz extension of $\phi^{-1}$.
\end{remark}

\begin{theorem}\label{(half) volume comparison}
Let $\Sigma$ be a cap and $\Gamma_1,\Gamma_2\subset \Sigma$ be two convex-like submanifolds. If $\Gamma_1$ is inside $\Gamma_2$ on $\Sigma$, then
\[
V^{n-1}(\Gamma_2) \geq V^{n-1}(\Gamma_1).
\]

\end{theorem}

\begin{proof}
We only consider $(n-1)$-dimensional Jacobians. We use $J_L$ for Lorentzian Jacobians and $J$ for the usual Jacobians between unit spheres.

Let $P_i : \Gamma_i \rightarrow S^{n-1}$ be Lipschitz homeomorphisms and $X_i = P_i^{-1}$. It suffices to show that
\[
\int_{S^{n-1}} J_LX_1({y}) d\mathcal{H}^{n-1}({y}) \leq \int_{S^{n-1}} J_LX_2({x}) d\mathcal{H}^{n-1}({x}).
\]
Write $a_i=J_LX_i$ and set $a_i=0$ at points where $X_i$ is not differentiable. These functions are bounded, and this convention does not change the volume integrals.

By Lemma~\ref{injection enlarging distance}, there exists an injection $\phi : \Gamma_1 \rightarrow \Gamma_2$ such that
\[
|\phi(y_1) - \phi(y_2)|_L \geq |y_1-y_2|_L,\forall y_1,y_2 \in \Gamma_1.
\]

Let $G = \phi(\Gamma_1)$ and $T = P_2(G)$. Then $\phi^{-1} : G \rightarrow \Gamma_1$ satisfies
\begin{equation}\label{inq distance}
    |\phi^{-1}(x_1) - \phi^{-1}(x_2)|_L \leq |x_1-x_2|_L.
\end{equation}

By Theorem~\ref{uniformly spacelike}, $\hat{\phi} =P_1 \circ \phi^{-1} \circ P_2^{-1}$ is a Lipschitz surjective map from $T \subset S^{n-1}$ to $S^{n-1}$. Let ${F} : S^{n-1} \rightarrow S^{n-1}$ be its Lipschitz extension (Corollary~\ref{lip extension on sphere}).

Let $E=\overline{T}$. Then $E$ is compact and $ F(E)=S^{n-1}$. By continuity, (\ref{inq distance}) implies
\begin{equation}\label{inq distance E}
    |X_1\circ F(a)-X_1\circ F(b)|_L
\leq |X_2(a)-X_2(b)|_L,\qquad a,b\in E.
\end{equation}

We claim that
\begin{equation}\label{jacobian inq}
a_1(F(x))\,JF(x)\le a_2(x)
\qquad\text{for $\mathcal H^{n-1}$-almost every }x\in E.
\end{equation}
Indeed, almost every $x\in E$ is a density-one point of $E$
at which both $F$ and $X_2$ are differentiable.
If $X_1$ is not differentiable at $F(x)$, then
$a_1(F(x))=0$, so the claim is immediate.

Otherwise, fix $v\in T_xS^{n-1}$. The density-one property allows
us to choose $t_k\downarrow0$ and $x_k\in E$ such that
\[
\exp_x^{-1}(x_k)=t_kv+o(t_k).
\]
For $v\ne0$, failure of this approximation would leave balls
of radius comparable to their distance from $x$ disjoint from $E$,
contradicting density one. The case $v=0$ is immediate.
Taking $y=x$ in \eqref{inq distance E}, dividing by $t_k^2$,
and passing to the limit gives
\[
|D(X_1\circ F)_xv|_L^2\le |D(X_2)_xv|_L^2.
\]

In Euclidean orthonormal bases of $T_xS^{n-1}$ and $T_{F(x)}S^{n-1}$,
let $B$ be the matrix of $DF_x$, and let $G_2,G_1$ be the
Lorentzian Gram matrices of $D(X_2)_x,D(X_1)_{F(x)}$, respectively.
These Gram matrices are positive definite by uniform spacelikeness
and the bi-Lipschitz property of the parameterizations.
The preceding inequality says
\[
B^{\mathsf T}G_1B\preceq G_2.
\]
Taking determinants and square roots yields
\[
a_1(F(x))\,JF(x)
=\sqrt{\det G_1}\,|\det B|
\le \sqrt{\det G_2}
=a_2(x).
\]
This proves \eqref{jacobian inq}.

Finally, apply Corollary~\ref{lip area formula sphere}
with the bounded nonnegative weight
$\mathbf 1_E(a_1\circ F)$. Since $F(E)=S^{n-1}$, every point of
$S^{n-1}$ has at least one preimage in $E$. Consequently, by the area formula,
\begin{align*}
V^{n-1}(\Gamma_1)
&=\int_{S^{n-1}}a_1(y)\,d\mathcal H^{n-1}(y)\\
&\le
\int_{S^{n-1}}
\sum_{\substack{x\in E\\F(x)=y}}
a_1(y)\,d\mathcal H^{n-1}(y)\\
&=\int_E a_1(F(x))\,JF(x)\,d\mathcal H^{n-1}(x)\\
&\le\int_E a_2(x)\,d\mathcal H^{n-1}(x)\\
&\le\int_{S^{n-1}}a_2(x)\,d\mathcal H^{n-1}(x)
=V^{n-1}(\Gamma_2).
\end{align*}

\end{proof}

Theorem \ref{volume comparison} immediately follows from Theorem~\ref{(half) volume comparison} and Proposition~\ref{insert a convex-like submanifold}.

\section{Distance and convergence}


We use $h(A,B)$ to represent the one-way Hausdorff distance, i.e., for $A,B\subset \mathbb{R}_1^{n+1}$,
\[
h(A,B) = \sup_{x\in A}\mathrm{dist}(x,B) = \sup_{x\in A} ( \inf_{y\in B} \mathrm{dist}(x,y));
\]
and
\[
H(A,B) = \max \{h(A,B),h(B,A)\}
\]
denotes the Hausdorff distance between $A$ and $B$.

Let $B_r$ be the Euclidean ball centered at the origin with radius $r$, and let $B(x,r)$ be the Euclidean ball centered at $x$ with radius $r$.

Denote by $+$ the sum of subsets, i.e. for $A,B\subset \mathbb{R}_1^{n+1}$,
\[
A +B = \{x+y|x\in A,y \in B\}.
\]

For background on support functions and the Hausdorff metric
for Euclidean convex bodies, see
\cite[Sections~1.7--1.8]{Schneider}.

Remark: Under a Lorentzian coordinate transform, the Euclidean length, Hausdorff distance and Euclidean ball changes. However, the topology on points given by Euclidean distance doesn't change, and the topology on subsets(in particular, convex-like submanifolds) given by Hausdorff distance doesn't change.

\begin{lemma}
Fix a Lorentzian coordinate. Suppose two face-up towers $\Omega^+$, $\Omega_1^+$ and a face-down tower $\Omega^-$ satisfy:

\[
\begin{aligned}
\text{(a)}\quad &\Omega_1^+ \subset \Omega^+;\\
\text{(b)}\quad &B_r \subset \Omega_1^+ \cap \Omega^- \subset \Omega^+ \cap \Omega^- \subset B_R;\\
\text{(c)}\quad &h(Cap_1^+ \cap B_R, Cap^+) = d \leq r/2.
\end{aligned}
\]

Then
\[
h(Cap_1^+ \cap Cap^-, Cap^+\cap Cap^-) \leq \frac{R^2}{r^2}d.
\]

\end{lemma}

\begin{proof}
Let $K_1=\Omega_1^+\cap\Omega^-$, $K=\Omega^+\cap\Omega^-$, $\Gamma_1=Cap_1^+\cap Cap^-$ and $\Gamma=Cap^+\cap Cap^-$. Since these towers contain $B_r$, the closure of each intersection equals the intersection of the corresponding closures. In particular, $\Gamma_1\subset\overline{K_1}\subset\overline B_R$.

We first explain why (c) gives $\mathrm{dist}(x,Cap^+)\leq d$ also when $x\in\Gamma_1\cap\partial B_R$. For $0<\varepsilon<1$, the point $(1-\varepsilon)x$ belongs to $K_1$. Move it in the future vertical direction until it reaches $Cap_1^+$, and denote the resulting point by $z_\varepsilon$. Every cap is the graph of a continuous function over the spatial hyperplane, so $z_\varepsilon\to x$. By Proposition~\ref{tower and timecone}, $z_\varepsilon\in\Omega^-$. A sufficiently short past-pointing null-generator segment from $z_\varepsilon$ stays in $Cap_1^+\cap\Omega^-$ by Propositions~\ref{caps' support representation} and~\ref{outward Nulls}. This segment lies in $\overline{K_1}\subset\overline B_R$, and its interior points lie in $B_R$ by strict convexity of the Euclidean ball. Thus $x\in\overline{Cap_1^+\cap B_R}$, and the desired bound follows by continuity of the distance function.

Fix $x\in\Gamma_1$. If $x\in\Gamma$, there is nothing to prove. Otherwise $x\in\Omega^+$. By Propositions~\ref{caps' support representation} and~\ref{outward Nulls}, there is a past-pointing null vector $L^-$ such that $x+(L^-)^\perp$ supports $\Omega^-$ and $x-sL^-\in Cap^-$ for all $s\geq0$. This ray must leave $\Omega^+$, since $\overline{\Omega^+}\cap\overline{\Omega^-}=\overline K$ is bounded. Its first exit point belongs to $\Gamma$. Rescaling $L^-$ by a positive factor, we may therefore write this point as $y=x-L^-$.

The supporting hyperplane gives $\langle z-x,L^-\rangle\geq0$ for every $z\in B_r$. Hence
\[
\langle x,L^-\rangle
\leq\inf_{z\in B_r}\langle z,L^-\rangle
=-r|L^-|_E.
\]
Set $t=R^2/(2r|L^-|_E)$ and $q=x+tL^-$. Since $|x|_E\leq R$, we obtain
\[
\langle q,q\rangle
=\langle x,x\rangle+2t\langle x,L^-\rangle
\leq R^2-2tr|L^-|_E=0.
\]
Moreover, $\langle q,L^-\rangle=\langle x,L^-\rangle<0$, so $q$ belongs to the past-pointing closed causal cone $\overline{\mathrm{Timecone}^-(O)}$. For every $z\in B_r\subset\Omega^+$, Proposition~\ref{tower and timecone} gives $z+q\in\Omega^+$; thus $q+B_r\subset\Omega^+$.

Because $y\in\partial\Omega^+$ and $\Omega^+$ is open and convex,
\[
(1-\lambda)y+\lambda(q+B_r)\subset\Omega^+,
\qquad 0<\lambda\leq1.
\]
Taking $\lambda=1/(1+t)$ and using $(ty+q)/(1+t)=x$, we get
\[
B\left(x,\frac{r}{1+t}\right)\subset\Omega^+.
\]
Consequently,
\[
d\geq\mathrm{dist}(x,Cap^+)
\geq\frac{r}{1+t}
=\frac{r}{1+\frac{R^2}{2r|L^-|_E}},
\]
and therefore
\[
\mathrm{dist}(x,\Gamma)\leq|x-y|_E=|L^-|_E
\leq\frac{R^2d}{2r(r-d)}
\leq\frac{R^2}{r^2}d.
\]
Taking the supremum over $x\in\Gamma_1$ proves the assertion.
\end{proof}

\begin{corollary}\label{nesting distance control}
Fix a Lorentz coordinate system. Suppose two face-up towers $\Omega_1^+$, $\Omega_2^+$ and two face-down towers $\Omega_1^-$ and $\Omega_2^-$ satisfy:

\[
\begin{aligned}
\text{(a)}\quad &B_r \subset \Omega_1^+ \cap \Omega_1^- \subset \Omega_2^+ \cap \Omega_2^- \subset B_R;\\
\text{(b)}\quad &h(Cap_1^\pm \cap B_R, Cap_2^\pm) \leq d \leq r/2.
\end{aligned}
\]

Then
\[
h(Cap_1^+ \cap Cap_1^-, Cap_2^+\cap Cap_2^-) \leq \frac{2R^2}{r^2}d.
\]

\end{corollary}

\begin{proof}
Put $K_i=\Omega_i^+\cap\Omega_i^-$ and $\Gamma_i=Cap_i^+\cap Cap_i^-$ for $i=1,2$, and define
\[
A^\pm=\Omega_1^\pm\cap\Omega_2^\pm.
\]
These are towers with support functions $\max\{M_1^\pm,M_2^\pm\}$. Since $K_1\subset K_2$, we have
\[
A^+\cap A^-=K_1,
\qquad
B_r\subset K_1\subset\Omega_2^+\cap A^-\subset K_2\subset B_R.
\]
Let
\[
\Gamma_A=\partial A^+\cap\partial A^-,
\qquad
\Gamma'=Cap_2^+\cap\partial A^-.
\]
Both are convex-like submanifolds by Theorem~\ref{intersection of caps}. Moreover, $\Gamma_1\subset\Gamma_A$: indeed, every $x\in\Gamma_1$ lies in $\overline K_1\subset\overline K_2$ and on $Cap_1^\pm$, and hence on both $\partial A^+$ and $\partial A^-$.

As all the towers contain $B_r$, we also have
\[
\partial A^\pm\subset Cap_1^\pm\cup Cap_2^\pm.
\]
Thus assumption (b) implies
\[
h(\partial A^\pm\cap B_R,Cap_2^\pm)\leq d.
\]
Apply the preceding lemma first by enlarging $A^+$ to $\Omega_2^+$ while keeping $A^-$ fixed. Then apply its time-reversed version by enlarging $A^-$ to $\Omega_2^-$ while keeping $\Omega_2^+$ fixed. The inner and outer radii remain $r$ and $R$, so
\[
h(\Gamma_A,\Gamma')\leq\frac{R^2}{r^2}d,
\qquad
h(\Gamma',\Gamma_2)\leq\frac{R^2}{r^2}d.
\]
Since $\Gamma_1\subset\Gamma_A$, the triangle inequality for the one-way Hausdorff distance gives
\[
h(\Gamma_1,\Gamma_2)
\leq h(\Gamma_A,\Gamma')+h(\Gamma',\Gamma_2)
\leq\frac{2R^2}{r^2}d.
\]
\end{proof}

\begin{lemma}\label{enlarging a cap}
Fix a Lorentz coordinate system. Suppose two face-up towers $\Omega_2^+$ and $\Omega_1^+$ satisfy:

\begin{enumerate}
    \item $B_r \subset \Omega_2^+ $;
    \item $h(Cap_1^+, Cap_2^+) \leq d$, where $d>0$.
\end{enumerate}

Then $\Omega_1^+ \subset (1+d/r)\Omega_2^+$.
\end{lemma}

\begin{proof}
\begin{align*}
    \Omega_1^+ &\subset \Omega_2^+ + B_d\\
    &\subset \Omega_2^+ + (d/r)\Omega_2^+\\
    &= (1+d/r)\Omega_2^+.
\end{align*}
\end{proof}

\begin{theorem}\label{one-way distane control}

Fix a Lorentz coordinate system and positive numbers $d,r,R$, where $0<r<R$ and
\[
\frac{Rd}{r} + d \leq \frac{r^2}{2(r+d)}.
\]
Suppose two face-up towers $\Omega_1^+$, $\Omega_2^+$ and two face-down towers $\Omega_1^-$, $\Omega_2^-$ satisfy:

\[
\begin{aligned}
\text{(a)}\quad &B_r \subset \Omega_1^+ \cap \Omega_1^- \subset B_R
  \text{ and } B_r \subset \Omega_2^+ \cap \Omega_2^- \subset B_R;\\
\text{(b)}\quad &h(Cap_1^+, Cap_2^+) \leq d
  \text{ and } h(Cap_1^-, Cap_2^-) \leq d.
\end{aligned}
\]

Then
\[
h(Cap_1^+ \cap Cap_1^-, Cap_2^+\cap Cap_2^-)
\leq \frac{Rd}{r+d} + \frac{2R^2(r+d)^2}{r^4}\left(d + \frac{Rd}{r}\right).
\]
\end{theorem}

\begin{proof}
By Lemma~\ref{enlarging a cap}, applied to both pairs of towers (with time-orientation reversed for the face-down towers),
\[
\Omega_1^+ \cap \Omega_1^- \subset (1+d/r)(\Omega_2^+ \cap \Omega_2^-).
\]
Let
\[
s=\frac{r}{r+d},\qquad \rho=sr=\frac{r^2}{r+d},\qquad \delta=d+\frac{Rd}{r}.
\]
Then
\[
B_\rho\subset s(\Omega_1^+\cap\Omega_1^-)
\subset\Omega_2^+\cap\Omega_2^-\subset B_R.
\]

Let $\Gamma_1 = Cap_1^+ \cap Cap_1^-$ and $\Gamma_2 = Cap_2^+\cap Cap_2^-$. We have
\[
h(\Gamma_1, \Gamma_2) \leq h(\Gamma_1, s\Gamma_1) + h(s\Gamma_1,\Gamma_2) \leq \frac{Rd}{r+d} + h(s\Gamma_1,\Gamma_2).
\]

Moreover, for $x\in(sCap_1^\pm)\cap B_R$, the point $x/s$ lies in $Cap_1^\pm$ and

\[
|x/s-x|_E\leq(1/s-1)R=Rd/r.
\]
Thus
\begin{align*}
h((sCap_1^\pm)\cap B_R,Cap_2^\pm)
&\leq h((sCap_1^\pm)\cap B_R,Cap_1^\pm)+h(Cap_1^\pm,Cap_2^\pm)\\
&\leq\frac{Rd}{r}+d=\delta.
\end{align*}

Since $\delta\leq\rho/2$, Corollary~\ref{nesting distance control}, with inner radius $\rho$, gives
\[
h(s\Gamma_1,\Gamma_2)\leq\frac{2R^2}{\rho^2}\delta
=\frac{2R^2(r+d)^2}{r^4}\left(d+\frac{Rd}{r}\right).
\]
So
\[
h(\Gamma_1,\Gamma_2)
\leq\frac{Rd}{r+d}+h(s\Gamma_1,\Gamma_2)
\leq\frac{Rd}{r+d}+\frac{2R^2(r+d)^2}{r^4}\left(d+\frac{Rd}{r}\right).
\]
\end{proof}

\begin{theorem}\label{continuity of caps' intersetion}
Given two continuous homogeneous functions $M^+:\mathrm{Nullcone}^+ \rightarrow R$ and $M^-:\mathrm{Nullcone}^- \rightarrow R$ where $\Omega^+(M^+) \cap \Omega^-(M^-) \neq \emptyset$. Let $\Gamma = Cap^+(M^+) \cap Cap^-(M^-)$.

If a sequence of continuous homogeneous functions $\{M_k^+:\mathrm{Nullcone}^+ \rightarrow R\}$ uniformly converges to $M^+$ in any bounded region, and another sequence of continuous homogeneous functions $\{M_k^-:\mathrm{Nullcone}^- \rightarrow R\}$ uniformly converges to $M^-$ in any bounded region. Let $\Gamma_k= Cap^+(M_k^+) \cap Cap^-(M_k^-)$. Then
\[
H(\Gamma_k,\Gamma) \rightarrow 0.
\]
\end{theorem}

\begin{proof}
Fix a coordinate and let $e_t = (0,...,0,1)$. We define

\[\frac{d_k}{\sqrt{2}} = \sup \{\sup_{L^+ \in \mathrm{Nullcone}^+} \frac{|M^+(L^+)-M_k^+(L^+)|}{|L^+|_E},\sup_{L^- \in \mathrm{Nullcone}^-} \frac{|M^-(L^-)-M_k^-(L^-)|}{|L^-|_E}\}.\]

Then $d_k\to0$ by homogeneity and uniform convergence on the Euclidean-unit null directions.

For $x\in Cap^+(M_k^+)$, the definition of $d_k$ gives $x+d_ke_t\notin\Omega^+(M^+)$ and $x-d_ke_t\in\overline{\Omega^+(M^+)}$. The segment joining these two points therefore meets $Cap^+(M^+)$, so $\mathrm{dist}(x,Cap^+(M^+))\leq d_k$ and we conclude
\[
h(Cap^+(M_k^+),Cap^+(M^+))\leq d_k.
\]

Similarly we deduce
\[
H(Cap^+(M^+),Cap^+(M_k^+))\leq d_k
\]
and
\[
H(Cap^-(M^-),Cap^-(M_k^-))\leq d_k.
\]

After a translation, choose $r,R>0$ such that
\[
B_{2r}\subset\Omega\subset B_{R/2},
\qquad \Omega=\Omega^+(M^+)\cap\Omega^-(M^-),
\]
and write $\Omega_k=\Omega^+(M_k^+)\cap\Omega^-(M_k^-)$. For all sufficiently large $k$, we have $d_k<r$. Since $B_{2r}\subset\Omega$, the definition of $d_k$ gives, for either time-orientation,
\[
M_k^\pm(L)\leq M^\pm(L)+\frac{d_k}{\sqrt{2}}|L|_E
\leq-\left(2r-\frac{d_k}{\sqrt{2}}\right)|L|_E<-r|L|_E,
\]
and hence $B_r\subset\Omega_k$. Applying Lemma~\ref{enlarging a cap} to both pairs of towers, with inner radius $2r$, we also obtain
\[
\Omega_k\subset\left(1+\frac{d_k}{2r}\right)\Omega
\subset B_{(1+d_k/(2r))R/2}\subset B_R.
\]
Thus both $\Omega$ and $\Omega_k$ contain $B_r$ and are contained in $B_R$. In particular, by Theorem~\ref{intersection of caps}, $\Gamma_k$ is a convex-like submanifold for all sufficiently large $k$.

If $d_k=0$, then $\Gamma_k=\Gamma$. Otherwise, applying Theorem~\ref{one-way distane control} in both directions, whenever
\[
\frac{Rd_k}{r}+d_k\leq\frac{r^2}{2(r+d_k)},
\]
we obtain
\[
H(\Gamma_k,\Gamma)\leq\frac{Rd_k}{r+d_k}
+\frac{2R^2(r+d_k)^2}{r^4}\left(d_k+\frac{Rd_k}{r}\right).
\]
Since $d_k\to0$, this condition holds for all sufficiently large $k$, and the right-hand side tends to zero. Hence $H(\Gamma_k,\Gamma)\to0$.
\end{proof}

As an application of volume comparison theorem, we can prove that the volume function is continuous in the topology induced by Hausdorff distance.

\begin{lemma}\label{enlarging submanifolds}

If $\Gamma$, $\Gamma_1$ are convex-like manifolds, with $h(\Gamma_1, \Gamma) \leq d$ and $B_r \subset \Omega(\Gamma)$. Then
\[
(1+\frac{d}{r})^{n-1}V(\Gamma) \geq V(\Gamma_1).
\]
\end{lemma}

\begin{proof}
Fix $\eta>0$. Since $\Gamma\subset\overline{\Omega(\Gamma)}$ and $h(\Gamma_1,\Gamma)\leq d$, we have
\begin{align*}
\Gamma_1
&\subset \Omega(\Gamma)+B_{d+\eta}\\
&\subset \Omega(\Gamma)+\frac{d+\eta}{r}\Omega(\Gamma)\\
&=\left(1+\frac{d+\eta}{r}\right)\Omega(\Gamma).
\end{align*}
Hence $\Gamma_1$ nests in $\left(1+\frac{d+\eta}{r}\right)\Gamma$. By Theorem~\ref{volume comparison} and the scaling of $(n-1)$-dimensional volume,
\[
V(\Gamma_1)\leq\left(1+\frac{d+\eta}{r}\right)^{n-1}V(\Gamma).
\]
Letting $\eta\downarrow0$ proves the result.
\end{proof}

\begin{theorem}
For convex-like submanifolds $\Gamma$ and $\Gamma_k$, if $\Gamma_k \to \Gamma$ then
\[
V(\Gamma_k) \to V(\Gamma).
\]
\end{theorem}

\begin{proof}
Let $\varepsilon_k=H(\Gamma_k,\Gamma)\to0$. After a translation, choose $r>0$ such that $B_{2r}\subset\Omega(\Gamma)$.

For a future- or past-pointing null vector $L$ with $|L|_E=1$, write
\[
M_k(L)=\inf_{p\in\Gamma_k}\langle p,L\rangle,
\qquad M(L)=\inf_{p\in\Gamma}\langle p,L\rangle.
\]
The linear function $p\mapsto\langle p,L\rangle$ is $1$-Lipschitz with respect to the Euclidean distance. Hence
\[
|M_k(L)-M(L)|\leq\varepsilon_k.
\]
Since $B_{2r}\subset\Omega(\Gamma)$, we have $M(L)\leq-2r$. For all sufficiently large $k$, $\varepsilon_k<r$, so for every $x\in B_r$,
\[
M_k(L)\leq-2r+\varepsilon_k<-r\leq\langle x,L\rangle.
\]
By homogeneity and the definitions of the two towers, this gives
\[
B_r\subset\Omega(\Gamma_k)\cap\Omega(\Gamma).
\]

Apply Lemma~\ref{enlarging submanifolds} first to $(\Gamma_1,\Gamma)=(\Gamma_k,\Gamma)$ and then with the roles of $\Gamma$ and $\Gamma_k$ interchanged. We obtain
\[
\left(1+\frac{\varepsilon_k}{r}\right)^{-(n-1)}V(\Gamma)
\leq V(\Gamma_k)
\leq\left(1+\frac{\varepsilon_k}{r}\right)^{n-1}V(\Gamma).
\]
Letting $k\to\infty$ proves $V(\Gamma_k)\to V(\Gamma)$.
\end{proof}

\section{Projection volume ratio}

Due to the negative metric along the timelike direction, a timelike orthogonal projection transforms every connected, smooth closed spacelike submanifold of codimension two into an immersed hypersurface with no less volume. In this section, the volume of an immersed projection is counted with multiplicity. If we fix one of such submanifolds $\Gamma$, then each timelike orthogonal projection $P_v$ gives a ratio
\[
\frac{V^{n-1}(P_v(\Gamma))}{V^{n-1}(\Gamma)} \geq 1.
\]
This ratio could be arbitrarily large (when $v$ is sufficiently close to a null vector); however we can consider its minimum (among all timelike orthogonal projections), and call
{\postdisplaypenalty=10000
\[
\xi(\Gamma) =\min_{\langle v,v\rangle < 0}\frac{V^{n-1}(P_v(\Gamma))}{V^{n-1}(\Gamma)}
\]
the projection volume ratio of $\Gamma$.}

To justify the minimum, normalize the projection direction to a future-pointing unit timelike vector $u$. If $u^{T\Gamma}$ denotes its tangential component, the Gram matrix of $P_u$ on an orthonormal tangent frame gives
\[
J(P_u|_\Gamma)=\sqrt{1+|u^{T\Gamma}|_L^2},\qquad
V^{n-1}(P_u(\Gamma))=\int_\Gamma\sqrt{1+|u^{T\Gamma}|_L^2}\,dV.
\]
This is a continuous function of $u$. If $|u_k|_E\to\infty$, pass to a subsequence such that $u_k/|u_k|_E\to L$, where $L$ is a nonzero future-pointing null vector. Then
\[
\frac{V^{n-1}(P_{u_k}(\Gamma))}{|u_k|_E}
\longrightarrow\int_\Gamma|L^{T\Gamma}|_L\,dV>0.
\]
Indeed, if the integral vanished, the function $x\mapsto\langle x,L\rangle$ would be constant on the connected manifold $\Gamma$, so $\Gamma$ would lie in a null hyperplane. Projection of that hyperplane along its null direction onto an $(n-1)$-dimensional spacelike subspace would restrict to a local diffeomorphism on $\Gamma$. Its image would be both open and compact, which is impossible. Thus the projection volume tends to infinity as $u$ leaves every compact subset of the future unit timelike hyperboloid, and its minimum is attained.

The same Jacobian formula shows that $\xi(\Gamma) \geq 1$, with equality if and only if a minimizing $u$ is orthogonal to every tangent space of $\Gamma$. By connectedness, this is equivalent to $\Gamma$ lying in a spacelike hyperplane. So we may consider $\xi$ as a function that describes how close a submanifold is to being on a hyperplane.

In general, $\xi$ can take arbitrarily large values. For example, consider a closed spacelike curve $r$ in $\mathbb{R}^3_1$:
\[
r(\theta) = (\cos \theta, \sin \theta , h(\theta))
\]
with $h:S^1 \to \mathbb{R}$, $|h'|<1$ and $|h| \leq 1/4$.

Its length
\[
V^1(r) = \int_{S^1}\sqrt{1-(h'(\theta))^2}d\theta
\]
could be arbitrarily close to $0$ (by properly choosing $h$). However, for any timelike orthogonal projection $P_v$, the curve $P_v(r)$ surrounds the projected disk
\[
P_v(D),\qquad D=\{(x,y,0)\mid x^2+y^2<1/4\},
\]
with winding number $1$. To see this, set $r_s(\theta)=(\cos\theta,\sin\theta,s h(\theta))$, $0\leq s\leq1$. For every $z\in D$, the vector $r_s(\theta)-z$ is spacelike: its horizontal length is greater than $1/2$, whereas its time component has absolute value at most $1/4$. Hence $P_v(r_s(\theta))\neq P_v(z)$. The homotopy $P_v(r_s)$ therefore preserves the winding number around every point of $P_v(D)$, and at $s=0$ that winding number is $1$, with the orientation induced by $P_v|_{\{z=0\}}$. Moreover, $P_v$ does not decrease lengths on the horizontal plane, so $P_v(D)$ contains the radius-$1/2$ disk centered at the origin in $v^\perp$. Integrating the absolute angular variation around this disk gives
\[
V^1(P_v(r))\geq\frac12\,2\pi=\pi.
\]

However, for smooth convex-like submanifolds, we have the following:

\begin{theorem}
Let $\Gamma$ be a smooth convex-like submanifold in $\mathbb{R}_1^{n+1}$. Then
\[
\xi(\Gamma) \leq n^{n-1}.
\]
\end{theorem}

\begin{proof}
For $p_1,...,p_{n+1} \in \Gamma$, the convex hull $\mathrm{Conv}(p_1,..,p_{n+1})$ has dimension at most $n$. So we can define
\begin{equation*}
A(p_1,...,p_{n+1}) =
\left\{\begin{aligned}
&\text{the $n$-volume of }\mathrm{Conv}(p_1,..,p_{n+1}),\\
&\text{\quad if the hull is spacelike and $n$-dimensional;}\\
&0,\quad\text{otherwise.}
\end{aligned}
\right.
\end{equation*}

In fact, there exist $n+1$ points $p_1,...,p_{n+1}\in \Gamma$ such that $\mathrm{span}\{p_2-p_1,...,p_{n+1}-p_1\}$ is a spacelike $n$-dimensional subspace.

To prove this, choose $p_1,p_{n+1}\in\Gamma$, such that $|p_1-p_{n+1}|_L$ attains its maximum. Then $p_1-p_{n+1}$ is spacelike.

By Theorem \ref{uniformly spacelike}, $T_{p_1}\Gamma$ is spacelike. By maximality of $|p_1-p_{n+1}|_L$, $p_1-p_{n+1}$ is perpendicular to $T_{p_1}\Gamma$, so $\mathrm{span}\{p_1-p_{n+1},T_{p_1}\Gamma\}$ is spacelike. Choosing $p_2,...,p_{n}$ close to $p_1$ with secant directions approaching a basis of $T_{p_1}\Gamma$ provides $\mathrm{span}\{p_2-p_1,...,p_{n+1}-p_1\}$ to be spacelike. This shows that $\sup A>0$.

$A$ is continuous. By compactness, there exist $n+1$ points $q_1,...,q_{n+1}$ where the positive maximum of $A$ is reached, and these points determine a unique spacelike hyperplane. Without loss of generality suppose $X^{n+1}(q_i) = 0$, $1\leq i\leq n+1$.

We take the orthogonal projection $P$ along $v = (0,...,0,1)$ to $\pi = \{x\mid X^{n+1}(x)=0\}$ and consider $\Gamma' = P(\Gamma) \subset \pi$ with $q_1,...,q_{n+1} \in \Gamma'$. Let $H = \mathrm{Conv}(q_1,...,q_{n+1})$ be a convex body in $\pi$ (so $\partial H$ is convex-like), and we want to show that
\[
V^{n-1}(\Gamma') \leq n^{n-1}V^{n-1}(\partial H).
\]

Let $F_i$ be the $(n-1)$-dimensional hyperplane in $\pi$ such that $q_j\in F_i$ for all $j\neq i$. Let $F'_i$ be the hyperplane parallel to $F_i$ through $q_i$. Let $H_i$ be the closed halfspace in $\pi$ with $\partial H_i = F_i$ and $q_i \in H_i$, and let $H'_i$ be the closed halfspace in $\pi$ with $\partial H'_i = F'_i$ and $q_j \in H'_i$ for all $j\neq i$.

Let $\delta$ be the homothety centered at the centroid of $\{q_1,...,q_{n+1}\}$ with ratio $-n$. Clearly $\delta(F_i) = F'_i$ and $\delta(H_i) = H'_i$. Hence
\[
H' =\bigcap_{1\leq i\leq n+1}H'_i=\delta(H)
\]
is a closed convex body $n$ times as large as $H = \bigcap_{1\leq i\leq n+1}H_i$.

To use volume maximality, fix the facet opposite $q_i$ and denote its $(n-1)$-volume by $B_i$. Write the normal coordinates of a nearby moving vertex as $(h,\tau)$, where $h$ is signed distance from $F_i$ within $\pi$, chosen positive at $q_i$, and $\tau=X^{n+1}$. The simplex volume near $q_i$ is
\[
\frac{B_i}{n}\sqrt{h^2-\tau^2}.
\]
At $q_i$ we have $\tau=0$ and $h>0$, so vanishing of its first variation along $\Gamma$ gives $dh|_{T_{q_i}\Gamma}=0$. In particular,
\[
T_{q_{n+1}}\Gamma\subset
\mathrm{span}\{q_2-q_1,...,q_n-q_1,(0,...,0,1)\},
\]
and
\[
T_{q_{n+1}}\Gamma'=\mathrm{span}\{q_2-q_1,...,q_n-q_1\}
\]
is parallel to $F_{n+1}$. The same argument applies to every vertex. Thus the tangent hyperplane of the convex hypersurface $\Gamma'\subset\pi$ at $q_i$ is exactly $F'_i$. Since the other vertices lie in $H'_i$, convexity gives $\Gamma'\subset H'_i$ for every $i$, and hence $\Gamma'\subset H'$. By Euclidean volume comparison,
\[
V^{n-1}(\Gamma') \leq V^{n-1}(\partial H') = n^{n-1}V^{n-1}(\partial H).
\]

Notice that
\[
\partial H\subseteq H=\mathrm{Conv}(q_1,...,q_{n+1})
\subseteq\mathrm{Conv}(\Gamma)\subseteq\overline{\Omega(\Gamma)}.
\]
Thus $\partial H$ nests in $\Gamma$, and the volume comparison theorem applies. Then we conclude
\[
V^{n-1}(\Gamma) \leq V^{n-1}(\Gamma') \leq n^{n-1}V^{n-1}(\partial H) \leq n^{n-1}V^{n-1}(\Gamma).
\]
\end{proof}

\appendix

\section{Lipschitz Mappings}

We collect some classic theorems about Lipschitz functions and mappings.

Let $\mathbb{R}^n$ be the $n$-dimensional real linear space endowed with the metric function $d_{R^n}(x,y) = |x-y|_E$, and $S^n \subset R^{n+1}$ be the unit sphere in $\mathbb{R}^{n+1}$ with the metric function $d_{S^n}(x,y) = d_{R^{n+1}}(x,y)$. Notice that $d_{S^n}$ is equivalent to the intrinsic distance with respect to its inherited Riemann metric.

\begin{definition}
Let $X,Y$ be metric spaces and $d_X,d_Y$ be their metric function. We say $f:X\to Y$ is Lipschitz if there exists $C\in R$ such that
\[
d_Y(f(x_1),f(x_2)) \leq C d_X(x_1,x_2),\,\forall x_1,x_2\in X.
\]

We define the Lipschitz constant of $f$ by:
\[
\mathrm{Lip}(f) = \sup_{x_1\neq x_2} \frac{d_Y(f(x_1),f(x_2))}{d_X(x_1,x_2)}.
\]
\end{definition}

\begin{lemma}\label{central projection lip}
Let $a\in \mathbb{R}^{n+1}$, $|a|< 1$ and let $r>0$. Then the central projection
\[P(x) = \{a+t(x-a)|t>0\}\cap S^n
\]
from $\mathbb{R}^{n+1}\backslash B(a,r)$ to $S^n$ is Lipschitz.
\end{lemma}

\begin{proof}
Put $s=|a|<1$. For $x,y\notin B(a,r)$, write $x=a+\alpha u$ and $y=a+\beta v$, where $u,v\in S^n$ and $\alpha,\beta\geq r$. Write
\[
P(x)=a+\lambda(u)u,\qquad P(y)=a+\lambda(v)v.
\]
By the cosine law applied to the triangle with vertices $0,a,P(x)$,
\[
\lambda(u)^2+2\langle a,u\rangle_E\lambda(u)+s^2=1.
\]
Therefore
\[
\lambda(u)=-\langle a,u\rangle_E+\sqrt{\langle a,u\rangle_E^2+1-s^2},
\qquad 1-s\leq\lambda(u)\leq1+s.
\]
For real numbers $b,c$, rationalizing the difference gives
\[
\left|\sqrt{b^2+1-s^2}-\sqrt{c^2+1-s^2}\right|
=\frac{|b-c|\,|b+c|}{\sqrt{b^2+1-s^2}+\sqrt{c^2+1-s^2}}
\leq|b-c|.
\]
It follows that
\[
|\lambda(u)-\lambda(v)|\leq2|\langle a,u-v\rangle_E|\leq2s|u-v|.
\]
Applying the cosine law to the two triangles with common vertex $a$, we obtain
\begin{align*}
|P(x)-P(y)|^2
&=(\lambda(u)-\lambda(v))^2+\lambda(u)\lambda(v)|u-v|^2\\
&\leq\bigl(4s^2+(1+s)^2\bigr)|u-v|^2,\\
|x-y|^2
&=(\alpha-\beta)^2+\alpha\beta|u-v|^2
\geq r^2|u-v|^2.
\end{align*}
Consequently,
\[
|P(x)-P(y)|\leq\frac{\sqrt{4s^2+(1+s)^2}}{r}|x-y|,
\]
which proves the assertion.
\end{proof}

\begin{theorem}\label{ap derivative}
\textbf{\emph{\cite[3.1.4]{Federer}.}} If $f:R^m \to R^n$ is $\mathcal{L}^m$ measurable, then
$A_i=\mathrm{dmn\, ap}\, D_i f$ is an $\mathcal{L}^m$ measurable set,
$ap\, D_i f$ is an $\mathcal{L}^m \llcorner A_i$ measurable function.
\end{theorem}

\begin{theorem}\label{lip differentiable a.e.: rademacher's theorem}
\textbf{\emph{Rademacher's theorem (\cite[3.1.6]{Federer}.}} If $f: R^m\to R^n$ is Lipschitz, then $f$ is differentiable $\mathcal{L}^m$ almost everywhere.

\end{theorem}

\begin{corollary}\label{lip differentiable a.e. on sphere}
Let $S^m$ be the unit sphere in $\mathbb{R}^{m+1}$, with distance function $d(x,y) = |x-y|_{R^{m+1}}$. If $f: S^m\to R^n$ is Lipschitz, then $f$ is differentiable almost everywhere.
\end{corollary}

\begin{theorem}\label{lip extension on Euclidean spaces}
\textbf{\emph{Kirszbraun's theorem (\cite[2.10.43]{Federer}.}} If $S\subset R^m$ and $f: S \to R^n$ is Lipschitz, then $f$ has a Lipschitz extension $g: R^m \to R^n$ with $\mathrm{Lip} (g) = \mathrm{Lip} (f)$.
\end{theorem}

\begin{corollary}\label{lip extension on sphere}
Let $m\leq n$. If $T\subset S^m$ and $f: T \to S^n$ is Lipschitz, then $f$ has a Lipschitz extension $F: S^m \to S^n$.
\end{corollary}

\begin{proof}
If $T=\emptyset$, a constant map gives the required extension. Otherwise, regard $f$ as a Lipschitz map from $T\subset \mathbb{R}^{m+1}$ to $\mathbb{R}^{n+1}$. By Theorem~\ref{lip extension on Euclidean spaces}, it extends to a Lipschitz map on $\mathbb{R}^{m+1}$. Restricting this extension to $S^m$ gives a Lipschitz map $g:S^m\to \mathbb{R}^{n+1}$ extending $f$, with $\mathcal{H}^{n+1}(g(S^m)) = 0$ and $g(S^m)$ compact. So there exists $p\in D^{n+1}$ and $0<r$ such that $B(p,r)\cap g(S^m) = \emptyset$.

Let
\[
F(x) = \{p+t(g(x)-p)|t>0\}\cap S^n
\]
be the central projection from $p$. Then by Lemma~\ref{central projection lip}, $F$ is Lipschitz, and $F|_T=f$ because $f(T)\subset S^n$.
\end{proof}

\begin{theorem}\label{lip area formula}
\textbf{\emph{(\cite[3.2.3(2)]{Federer}.}} Suppose $f: R^m \to R^n$ is Lipschitz with $m\leq n$. If $u$ is a $\mathcal{L}^m$ integrable function, then
\[
\int_{R^m} u(x) J_mf(x) d\mathcal{L}^m x = \int _{R^n}\sum_{x\in f^{-1}(y)} u(x) d\mathcal{H}^my.
\]
\end{theorem}

\begin{corollary}\label{lip area formula sphere}
Suppose $f: S^n \to S^n$ is Lipschitz. If $u$ is a $\mathcal{H}^n$ integrable function, then
\[
\int_{S^n} u(x) J_nf(x) d\mathcal{H}^n x = \int _{S^n}\sum_{x\in f^{-1}(y)} u(x) d\mathcal{H}^n y.
\]
\end{corollary}

\Needspace{4\baselineskip}
\bibliographystyle{alpha}
\bibliography{references}

@article{Ye1,
  author   = {Ye, N. and Ma, X. and Wang, D.},
  title    = {The {Fenchel}-type inequality in the 3-dimensional {Lorentz} space and a {Crofton} formula},
  journal  = {Ann. Global Anal. Geom.},
  volume   = {50},
  number   = {3},
  year     = {2016},
  pages    = {249--259},
  mrnumber = {3554374}
}

@article{Ye2,
  author   = {Ye, N. and Ma, X.},
  title    = {Closed strong spacelike curves, {Fenchel} theorem and plateau problem in the 3-dimensional {Minkowski} space},
  journal  = {Chinese Ann. Math. Ser. B},
  volume   = {40},
  number   = {2},
  year     = {2019},
  pages    = {217--226},
  mrnumber = {3902060}
}

@article{Fukui,
  author   = {Fukui, T. and Honda, A. and Umehara, M.},
  title    = {Convexity of space-like projections of submanifolds with co-dimension 2 in {Lorentz-Minkowski} space},
  journal  = {C. R. Math. Acad. Sci. Paris},
  volume   = {363},
  year     = {2025},
  pages    = {109--113},
  mrnumber = {4876249}
}

@book{ONeill,
  author    = {O'Neill, B.},
  title     = {{Semi-Riemannian} geometry},
  series    = {Pure and Applied Mathematics},
  volume    = {103},
  publisher = {Academic Press},
  address   = {New York},
  year      = {1983},
  mrnumber  = {0719023}
}

@book{Schneider,
  author    = {Schneider, R.},
  title     = {Convex bodies: the {Brunn-Minkowski} theory},
  edition   = {Second expanded},
  series    = {Encyclopedia of Mathematics and its Applications},
  volume    = {151},
  publisher = {Cambridge University Press},
  address   = {Cambridge},
  year      = {2014},
  mrnumber  = {3155183}
}

@book{Federer,
  author    = {Federer, H.},
  title     = {Geometric measure theory},
  series    = {Die Grundlehren der mathematischen Wissenschaften},
  volume    = {153},
  publisher = {Springer-Verlag New York, Inc.},
  address   = {New York},
  year      = {1969},
  mrnumber  = {0257325}
}

@misc{AZ26,
      title={High codimension mean curvature flow of spacelike-convex submanifolds with one spacelike codimension}, 
      author={Ben Andrews and Qiyu Zhou},
      year={2026},
      eprint={2602.19537},
      archivePrefix={arXiv},
      primaryClass={math.DG},
      url={https://arxiv.org/abs/2602.19537}, 
      note=
      {arXiv:2602.19537}
}

@article{Van52,
  author  = {van Heijenoort, J.},
  title   = {On locally convex manifolds},
  journal = {Communications on Pure and Applied Mathematics},
  volume  = {5},
  number  = {3},
  pages   = {223--242},
  year    = {1952},
  doi     = {10.1002/cpa.3160050302}
}

@article{BW14,
  author  = {Brendle, Simon and Wang, Mu-Tao},
  title   = {A {Gibbons--Penrose} inequality for surfaces in
             {Schwarzschild} spacetime},
  journal = {Communications in Mathematical Physics},
  volume  = {330},
  number  = {1},
  pages   = {33--43},
  year    = {2014},
  doi     = {10.1007/s00220-014-1972-6}
}

@article{Li22,
  author  = {Li, Yang},
  title   = {{Dirichlet} problem for maximal graphs of higher codimension},
  journal = {International Mathematics Research Notices},
  volume  = {2022},
  number  = {3},
  pages   = {2159--2179},
  year    = {2022},
  doi     = {10.1093/imrn/rnaa099}
}

@misc{HM26,
  author       = {Hu, Baichuan and Ma, Xiang},
  title        = {Weighted normalized curve shortening flow
                  with applications in pseudo-{Euclidean} spaces},
  year         = {2026},
  howpublished = {arXiv:2609.24027 [math.DG]},
  url          = {https://arxiv.org/abs/2609.24027}
}

\end{document}